\documentclass[11pt]
{amsart}
\usepackage{amssymb,amsmath,amsthm,amsfonts,amsopn,url,color,hyperref,enumerate,mathtools,microtype,MnSymbol}
\usepackage[normalem]{ulem}
\usepackage[all]{xy}
\input xy
\xyoption{all}
\usepackage{amscd}
\usepackage{soul}

\theoremstyle{plain}
\newtheorem{thm}{Theorem}[section]

\newtheorem*{thma}{Theorem A}
\newtheorem*{thmb}{Theorem B}
\newtheorem*{thmc}{Theorem C}
\newtheorem*{thmd}{Theorem D}
\newtheorem*{thme}{Theorem E}
\newtheorem*{thmf}{Theorem F}
\newtheorem*{thmg}{Theorem G}
\newtheorem{prop}[thm]{Proposition}
\newtheorem{lemma}[thm]{Lemma}
\newtheorem{cor}[thm]{Corollary}

\theoremstyle{definition}
\newtheorem{defn}[thm]{Definition}
\newtheorem*{defn*}{Definition}
\newtheorem*{question*}{Question}

\newtheorem{example}[thm]{Example}
\newtheorem*{example*}{Example}
\newtheorem{rem}[thm]{Remark}
\newtheorem*{rem*}{Remark}

\newcommand{\field}[1]{\mathbb{#1}}
\newcommand{\N}{\field{N}}

\newcommand{\ideal}[1]{\mathfrak{#1}}
\newcommand{\m}{\ideal{m}}
\newcommand{\n}{\ideal{n}}

\newcommand{\func}[1]{\mathrm{#1} \,}

\newcommand{\im}{\func{im}}

\newcommand{\bfx}{\mathbf{x}}

\newcommand{\ra}{\rightarrow}

\DeclareMathOperator{\len}{\lambda}

\DeclareMathOperator{\ann}{ann}

\DeclareMathOperator{\Hom}{Hom}

\newcommand{\be}{\begin{enumerate}}
\newcommand{\ee}{\end{enumerate}}

\newcommand{\cS}{\mathcal{S}}
\newcommand{\li}
 {\leftfootline}

\newcommand{\onto}{\twoheadrightarrow}

\newcommand{\into}{\hookrightarrow}

\newcommand{\cA}{\mathcal{A}}

\newcommand{\cM}{\mathcal{M}}
\newcommand{\cP}{\mathcal{P}}
\renewcommand{\phi}{\varphi}

\newcommand{\subsel}{submodule selector}

\newcommand{\dual}{\smallsmile}
\newcommand{\cl}{{\mathrm{cl}}}
\let\int\relax
\DeclareMathOperator{\int}{i}

\newcommand{\fg}{finitely generated}
\newcommand{\charp}{characteristic $p>0$}

\DeclareMathOperator{\tr}{tr}

\newcommand{\CM}{Cohen-Macaulay}

\DeclareMathOperator{\soc}{soc}
\newcommand{\tfae}{The following are equivalent:}
\DeclareMathOperator{\core}{-core}
\DeclareMathOperator{\hull}{-hull}

\newcommand{\clr}{complete Noetherian local ring}

\makeatletter
\def\@settitle{\begin{center}%
  \baselineskip14\p@\relax
  \bfseries
  \uppercasenonmath\@title
  \@title
  \ifx\@subtitle\@empty\else
     \\[1ex]
     \@subtitle
  \fi
  \end{center}%
}
\def\subtitle#1{\gdef\@subtitle{#1}}
\def\@subtitle{}
\makeatother

\author{Neil Epstein}
\address{Department of Mathematical Sciences \\ George Mason University \\ Fairfax, VA  22030}
\email{nepstei2@gmu.edu}

\author{Rebecca R.G.}
\address{Department of Mathematical Sciences \\ George Mason University \\ Fairfax, VA  22030}
\email{rrebhuhn@gmu.edu}

\author{Janet Vassilev}
\address{Department of Mathematics and Statistics \\ University of New Mexico \\ Albuquerque, NM 87131}
\email{jvassil@math.unm.edu }

\title[Core-hull duality]{Nakayama closures, interior operations, and core-hull duality}
\subtitle{with applications to tight closure theory}

\date{August 4, 2020}
\begin{document}
\begin{abstract}
Exploiting the interior-closure duality developed by Epstein and R.G. \cite{nmeRG-cidual}, we show that for the class of Matlis dualizable modules $\cM$ over a Noetherian local ring, when $\cl$ is a Nakayama closure and $\int$ its dual interior, there is a duality between $\cl$-reductions and $\int$-expansions that leads to a duality between the $\cl$-core of modules in $\cM$ and the $\int$-hull of modules in $\cM^\vee$.  We further show that many algebra and module closures and interiors are Nakayama and describe a method to compute the interior of ideals using closures and colons.  
We use our methods to give a unified proof of the equivalence of F-rationality with F-regularity, and of F-injectivity with F-purity, in the complete Gorenstein local case. Additionally, we give a new characterization of the finitistic tight closure test ideal in terms of 
maps from $R^{1/p^e}$.
Moreover, we show that the liftable integral spread of a module exists.
\end{abstract}

\maketitle

\setcounter{tocdepth}{1} 
\tableofcontents

\section{Introduction}\label{sec:intro}

The core of an ideal, attributed to Sally and Rees \cite{ReSa-core}, was initially studied in part because of its relationship to the Brian\c{c}on-Skoda Theorem. While the core is in general difficult to compute, Huneke and Swanson  gave a formula for computing the core of an $\m$-primary integrally closed ideal $I$ in a 2-dimensional regular local ring with infinite residue field \cite{HuSw-core}. Papers by Mohan \cite{Mo-coremod}, Corso, Polini, and Ulrich \cite{CPU-strcore,CPU-coreres,CPU-corpd1}, Hyry and Smith \cite{HySmi-cvgc,HySmi-Kawcore}, Huneke and Trung \cite{HT-core}, Polini and Ulrich \cite{PU-core}, and Fouli \cite{Fou-corchar}, and Fouli, Polini, and Ulrich \cite{FPU-corechar} further explored formulas for cores of ideals and modules.

Fouli and the third named author \cite{FoVa-core} extended the notion of core to Nakayama closures $\cl$, defining the  $\cl$-core of an ideal as the intersection of all $\cl$-reductions of the ideal. They compared the original core to tight closure and Frobenius closure cores, showing that $\text{core}(I) \subseteq *\core(I) \subseteq F\core(I)$. They also came up with conditions when the core and the $*\core$ are equal.  Fouli, Vassilev and Vraciu \cite{FoVaVr-*core} determined $*\core(I)=(J:I)J$ for a minimal reduction $J$ over a normal local ring of characteristic $p>0$ when certain colon conditions hold. 

They were also able to show under the same conditions that the $*\core$ of an ideal is an intersection of finitely many $*$-reductions of the ideal.

The first named author and Schwede introduced the tight interior in \cite{nmeSc-tint} as a dual to tight closure.  In \cite{nmeRG-cidual}, the first two named authors showed that this is an example of a more general duality, which assigns to each closure operation a corresponding interior operation.
The third named author noted in \cite{Va-inthull} that given an interior operation i, we can define the i-hull of an ideal in a way analogous but seemingly dual to the cl-core. 
In this paper we expand on this duality, ultimately proving that it arises from the closure-interior duality defined in \cite{nmeRG-cidual}.
More specifically, we have the following result: 

\begin{thma}
[Duality Theorem]
Let $\cl$ be a Nakayama closure on Noetherian $R$-modules, where $R$ is a complete Noetherian local ring. Then the dual interior operation to $\cl$, $\int$, is a Nakayama interior operation defined on Artinian $R$-modules (Proposition \ref{prop:nakayamaclosureint}). Let $A \subseteq B$ be Artinian $R$-modules.
\begin{enumerate}
    \item There exists an order-reversing correspondence between $\int$-expansions of $A \subseteq B$ and $\cl$-reductions of $(B/A)^\vee$ in $B^\vee$ (Theorem \ref{thm:expansiondualisreduction}).
    \item Every $\int$-expansion of $A$ in $B$ is contained in a maximal $\int$-expansion of $A$ in $B$ (Proposition \ref{prop:maxintexpansions}).
    \item The $\int$-hull of $A$ in $B$ exists and is dual to the $\cl$-core of $(B/A)^\vee$ in $B^\vee$ (Theorem \ref{thm:hullexists}).
\end{enumerate}
\end{thma}

In order to do this, we define Nakayama interiors (Definition \ref{def:Nakayamainterior}) and prove that they are dual to Nakayama closures (see Section \ref{sec:nakayama}). We define i-expansions of an ideal, and prove that there are maximal i-expansions (see Section \ref{sec:expansions}). This involves a discussion of co-generation of modules (see Definition \ref{def:cog}) as originally discussed in \cite{Vam}.

The best known examples of Nakayama closures on ideals are Frobenius closure, tight closure, and integral closure \cite{nme*spread, nme-sp}. These closures all extend to the module setting, integral closure by means of liftable integral closure \cite{nmeUlr-lint}, and our results apply to all of these.
We give additional examples of Nakayama closures and interiors in Section \ref{sec:manyNAK}. For instance, under a large variety of circumstances, module and algebra closures are Nakayama:

\begin{thmb}
[See Theorems~\ref{thm:fgNak} and \ref{thm:algclosuresareNak}, and Corollary~\ref{cor:F+bigNak}
]
Let $(R,\m)$ be a Noetherian local ring.  Let $S$ be an $R$-algebra with $\m S$ contained in the Jacobson radical of $S$. \begin{enumerate}
    \item $\cl_S$ is a Nakayama closure on finite $R$-modules. (Theorem~\ref{thm:algclosuresareNak})
    \item If $S$ is local and $N$ a \fg\ $S$-module, then $\cl_N$ is a Nakayama closure on \fg\ $R$-modules. (Theorem~\ref{thm:fgNak})
    \item In particular, Frobenius closure (in \charp) and plus closure (when $R$ is complete) are Nakayama closures. (Corollary~\ref{cor:F+bigNak})
    \item If $B$ is a local big \CM\ $R$-algebra and $R \ra B$ is a local homomorphism, then $\cl_B$ is a Nakayama closure. (Corollary~\ref{cor:F+bigNak})
\end{enumerate}
\end{thmb}

In addition to the duality theorem above, we show that interiors of ideals can be computed using closures and colons, particularly the \emph{Artinistic version} of the interior (analogous to the finitistic test ideal):

\begin{thmc}[See Theorem \ref{thm:finintideal}]
Let $(R,\m,k)$ be a complete Noetherian local ring that is approximately Gorenstein and let $\alpha_f$ denote the finitistic version of a functorial submodule selector $\alpha$. If $\{J_t\}$ is a decreasing nested sequence of irreducible ideals cofinal with the powers of $\m$, then for any ideal $I$,
\[(\alpha_f)^\dual(I)=\bigcap_{t \ge 0}(J_t:(J_t:I)_R^{\cl}).\]
\end{thmc}

We also show
that when $R$ is Gorenstein, if $R$  is $\cl$-rational then $R$ is $\cl$-regular (Corollary \ref{cor:irredseq}). Further, we use liftable integral closure \cite{nmeUlr-lint} to extend the notion of analytic spread to modules:

\begin{thmd}[See Theorem \ref{thm:lispread}]
Let $(R,\m,k)$ be a Noetherian local ring such that $k$ is infinite and $R$ is either of equal characteristic 0 or  reduced.  Then the liftable integral spread exists and agrees with the analytic spread in the sense of Eisenbud-Huneke-Ulrich.
\end{thmd}

Throughout the paper, we demonstrate the usefulness of our results by applying them to the case of tight closure. One such result is 
the following new characterization of the finitistic tight closure test ideal:

\begin{thme}[See Theorem \ref{thm:testidealfinitistic}]
Let $R$ be a complete Noetherian local $F$-finite reduced ring of prime characteristic $p>0$ and $c$ be a big test element for $R$.  The finitistic test ideal of $R$ consists of those elements $a\in R$ such that for every $\m$-primary ideal $J$ of $R$ and every nonnegative integer $e\geq 0$, there is an $R$-linear map $g: R^{1/{p^e}} \ra R/J$ with $g(c^{1/p^e}) = a+J$.
\end{thme}

As a consequence of Corollary \ref{cor:irredseq}, we get a unified proof of two results on $F$-rational and $F$-injective rings: 
\begin{thmf}[See Corollaries \ref{cor:Gorsop} and \ref{cor:Fproperties}]
Let $(R,\m,k)$ be a complete Gorenstein Noetherian local ring.  Let $\cl$ be a functorial and residual closure operation.  If there exists a system of parameters ${\mathbf x} := x_1, \ldots, x_d$ for which $J_t := (x_1^t, \ldots, x_d^t)$ are $\cl$-closed for infinitely many $t\in \N$,  then every ideal $I$ of $R$ is $\cl$-closed and every finitely generated $R$-module $M$ $0$ is $\cl$-closed in $M$ (Corollary \ref{cor:Gorsop}).  

As a consequence we recover the well known results (Corollary \ref{cor:Fproperties}):
  \begin{enumerate}
    \item\ If $R$ is equicharacteristic and $F$-rational, then it is $F$-regular.
    \item\ If $R$ is of prime characteristic $p>0$ and $F$-injective, then it is $F$-pure.
    \end{enumerate}
\end{thmf}

In Theorem \ref{thm:LScases}, we use results of Lyubeznik and Smith to describe cases where the tight interior and its Artinistic version agree, which enables us to compute tight interiors and hulls using Theorem \ref{thm:finintideal}.

\begin{thmg}[See Theorem~\ref{thm:LScases}]
Suppose that $(R,\m)$ is a complete reduced F-finite local ring and $I$ is an ideal of $R$. If $R$ has mild singularities or the pair $R,I$ arises from a graded situation, then the tight interior of $I$ and the Artinistic version of the tight interior of $I$ coincide.
\end{thmg}

 The organization of the paper is as follows:
 Section \ref{sec:background} gives relevant background. In Section \ref{sec:finint} we prove that we can compute interiors of ideals in Noetherian local rings using the dual closure and colons and we show that when our ring is Gorenstein, if parameter ideals are $\cl$-closed then all ideals are $\cl$-closed. In Section \ref{sec:manyNAK} we show that many module and algebra closures are Nakayama. In Section \ref{sec:nakayama} we define a Nakayma interior and show that Nakayama interiors are dual to Nakayama closures. In Section \ref{sec:expansions}, we discuss $\int$-expansions, finite co-generation, minimal co-generating sets, maximal expansions and the core-hull duality. In Section \ref{sec:chc} we compare related closures and interiors, define the co-spread as a dual to spread, and then show liftable integral spread exists. In Section \ref{sec:examples} we prove that in many cases of interest the tight interior and its Artinistic version coincide, which allows us to compute some examples of tight and Frobenius interiors and hulls. 

All rings will be assumed to be commutative unless otherwise specified.

\section{Background}
\label{sec:background}

We describe the duality between closure operations and interior operations over a \clr\ as first given by the first two named authors in \cite{nmeRG-cidual}. We recall the definition of a Nakayama closure, $\cl$-reductions, and the $\cl$-core, and give some of their properties.

\begin{defn}[{\cite{nmeRG-cidual}}]
Let $R$ be a ring, not necessarily commutative. Let $\cM$ be a class of (left) $R$-modules that is closed under taking submodules and quotient modules.  Let $\cP := \cP_\cM$ denote the set of all pairs $(L,M)$ where $M \in \cM$ and $L$ is a submodule of $M$ in $\cM$.

A \emph{\subsel} is a function $\alpha: \cM \rightarrow \cM$ such that \begin{itemize}
 \item $\alpha(M) \subseteq M$ for each $M \in \cM$, and
 \item for any isomorphic pair of modules $M, N \in \cM$ and any isomorphism $\phi: M \rightarrow N$, we have $\phi(\alpha(M)) = \alpha(\phi(M))$.
  \end{itemize}
  
An \emph{interior operation} is a submodule selector that is
\begin{itemize}
    \item \emph{order-preserving}, i.e. for any $L \subseteq M \in \cM$, $\alpha(L) \subseteq \alpha(M)$, and
    \item  \emph{idempotent}, i.e. for all $M \in \cM$, $\alpha(\alpha(M))=\alpha(M)$.
\end{itemize} 

A submodule selector $\alpha$ is \emph{functorial} if for any $g: M \rightarrow N$ in $\cM$, we have $g(\alpha(M)) \subseteq \alpha(N)$.

A \emph{closure operation} is an operation that sends each pair $(L,M) \in \cP$ to a module $L \subseteq L_M^\cl \subseteq M$ such that 
\begin{itemize}
    \item $\cl$ is \emph{idempotent,} i.e. $(L_M^\cl)_M^\cl=L_M^{\cl}$, and
    \item  $\cl$ is \emph{order-preserving on submodules,} i.e. if $L \subseteq N \subseteq M$, $L_M^\cl \subseteq N_M^\cl$.
\end{itemize} 

A closure operation is \emph{residual} if for any surjective map $q: M \twoheadrightarrow P$ in $\cM$, we have $(\ker q)^{\cl}_M = q^{-1}(0^{\cl}_P)$. Note that because $q$ is a surjection, we also have $q( (\ker q)_M^{\cl} ) = 0_P^{\cl}$.

A closure operation $\cl$ is \emph{finitistic} if for any $L \subseteq M$ for which $L^\cl_M$ is defined, and for any $z\in L^\cl_M$, there is some module $N$ with $L \subseteq N \subseteq M$ and $N/L$ \fg\ such that $z \in L^\cl_N$.

If $\cl$ is a closure operation, the \emph{finitistic version} $\cl_f$ of $\cl$ is given by \[
L^{\cl_f}_M = \bigcup \{L^\cl_N : L \subseteq N \subseteq M \text{ and } N/L \text{ is \fg}\}.
\]
\end{defn}

For the following definition and result, $R$ is a complete Noetherian local ring with maximal ideal $\m$, residue field $k$, and $E := E_R(k)$ the injective hull. We will use $^\vee$ to denote the Matlis duality operation. $\cM$ is a category of $R$-modules closed under  taking submodule and quotient modules, and such that for all $M \in \cM$, $M^{\vee\vee} \cong M$. For example, $\cM$ could be the category of \fg\ $R$-modules, or of Artinian $R$-modules.

\begin{defn}[{\cite{nmeRG-cidual}}]
Let $R$ be a complete Noetherian local ring.
Let $S(\cM)$ denote the set of all submodule selectors on $\cM$.  Define $\dual: S(\cM) \rightarrow S(\cM^\vee)$ as follows: For $\alpha \in S(\cM)$ and $M\in \cM^\vee$, \[
\alpha^\dual(M) := (M^\vee / \alpha(M^\vee))^\vee,
\]
considered as a submodule of $M$ in the usual way.  
\end{defn}

\begin{thm}[{\cite{nmeRG-cidual}}]
\label{thm:smsdual}
Let $r$ be a complete Noetherian local ring and
 $\alpha$ a submodule selector on $\cM$.  Then: \begin{enumerate}
 \item\label{it:duality} $(\alpha^\dual)^\dual = \alpha$,
 \item If $\alpha(M):=0_M^\cl$ for a residual closure operation $\cl$, then $\alpha^\dual$ is an interior operation. Conversely, if $\alpha$ is an interior operation, then $\alpha^\dual(M)=0_M^\cl$ for a uniquely determined residual closure operation $\cl$.
\end{enumerate}
\end{thm}

\begin{rem}
In consequence, if $\cl$ is a residual closure operation on $\cM$, its dual interior operation can be expressed as
\[\int(M)=\left(\frac{M^\vee}{0_{M^\vee}^{\cl}}\right)^\vee,\]
where $M \in \cM^\vee$.
\end{rem}

\begin{defn}
Let $R$ be a ring, not necessarily commutative, and let $L \subseteq M$ be $R$-submodules of $N \in \cM$.  We say that $L \subseteq N$ is a \emph{$\cl$-reduction of $N$ in $M$} if $L^{\cl}_M=N^{\cl}_M$.
\end{defn}

Note that $L \subseteq N \subseteq L_M^\cl$ if and only if $L$ is a $\cl$-reduction of $N$ in $M$.

\begin{rem}
\begin{enumerate}
\item If $N=R$ and $J \subseteq I$ are ideals of $R$ with $I^{\cl}=J^{\cl}$, this agrees with the notion of $J$ being a $\cl$-reduction of $I$.  
\item When $M=R$ is the ring, we will write $(-)^\cl$ in place of $(-)^\cl_R$.)  However, for a general $R$-module $M$, the closure of a submodule $N$ may change depending on the ambient $R$-module $M$.
Hence, we write $N^{\cl}_M$ to emphasize that we are taking the closure of $N$ in $M$.
\end{enumerate}
\end{rem}

\begin{defn}
Let $R$ be a ring, not necessarily commutative, and $\cl$ a closure operation defined on a class of $R$-modules $\cM$ that is closed under submodules, quotient modules, and extensions.  If $M,N$ are elements of $\cM$, with $N \subseteq M$, the \emph{$\cl$-core of $N$ with respect to $M$} is the intersection of all $\cl$ reductions of $N$ in $M$, or
\[\cl\core_M(N):= 
\bigcap_{L \subseteq N \subseteq L_M^{\cl}} L.\]
\end{defn}

\begin{defn}
\label{def:Nakayama}
Let $(R,\m)$ be a Noetherian local ring and $\cl$ be a closure operation on the class of \fg\ $R$-modules $\cM$.  We say that $\cl$ is a \emph{Nakayama closure} if for $L \subseteq N \subseteq M \in \cM$ satisfying $L\subseteq N \subseteq (L+\m N)^{\cl}_M$ implies that $L^{\cl}_M=N^{\cl}_M$.
\end{defn}

Note that this is consistent with the definition for ideals \cite{nme*spread} by letting $M=R/J$ and $L=I/J$ where $J \subseteq I$ are ideals. 
\begin{prop}\label{NR Ext}\cite{nme*spread, nme-sp}
Let $(R, \m)$ be a Noetherian local ring and $\cl$ a Nakayama closure operation on the class of \fg\ $R$-modules $\cM$. Let $N \subseteq M$ be elements of $\cM$. For any $\cl$-reduction $L$ of $N$ in $M$, there exists a minimal $\cl$-reduction $K$ of $N$ in $M$ such that $K \subseteq L$. Moreover, any minimal generating set of $K$ extends to one of $L$.
\end{prop}

\begin{rem}
The above result is stated only for ideals in \cite[Lemma 2.2]{nme*spread}, but the full proof of \cite[Lemma 2.2]{nme*spread} applies \emph{mutatis mutandis} to finitely generated modules (as is mentioned for part of the result on page 2210 of \cite{nme-sp}).
\end{rem}

The principal component of the above property holds for Artinian modules too, and in much greater generality.

\begin{prop}
Let $R$ be any associative ring (not necessarily commutative), let $C$ be an $R$-module, and let $\cl$ be a closure operation on (left) $R$-submodules of $C$, and let $A \subseteq B \subseteq C$ be $R$-modules with $A$ Artinian.  Suppose that $A$ is a $\cl$-reduction of $B$ in $C$.  Then there is a minimal $\cl$-reduction $K$ of $B$ in $C$, such that $K \subseteq A$.
\end{prop}

\begin{proof}
Let $\cS$ be the set of $R$-submodules $D$ of $A$ such that $D$ is a $\cl$-reduction of $B$ in $C$.  Note that $A \in \cS$, whence $\cS \neq \emptyset$.  Since $A$ is Artinian and $\cS$ is a nonempty collection of submodules of $A$, $\cS$ has a minimal element, say $K$.  We claim that $K$ is in fact a minimal $\cl$-reduction of $B$ in $C$, since if $E$ is a $\cl$-reduction of $B$ in $C$ and $E \subseteq K$, then in particular $E \subseteq A$, whence $E \in \cS$.  By minimality of $K$, then, we have $E=K$.  Hence $K$ is a minimal $\cl$-reduction of $B$ in $C$, and $K \subseteq A$.
\end{proof}
 
\section{The Artinistic interior of an ideal}\label{sec:finint}

In this section we prove a result on finitistic versions of closure operations that allows us to describe their dual interior operations. All rings are commutative unless otherwise specified.

\begin{defn}[{\cite[Definition 5.1]{nmeRG-cidual}}]\label{def:fin}
Let $\cl$ be a closure operation on a class $\cM$ of $R$-modules. We define the \emph{finitistic version $\cl_f$ of $\cl$} by
\[L^{\cl_f}_M:=\bigcup \{L_{N}^{\cl}: L \subseteq N \subseteq M \in \cM, N/L \text{ finitely-generated}\} .\]

Let $\alpha$ be a submodule selector on a class of $R$-modules. The \emph{finitistic version $\alpha_f$ of $\alpha$} is defined by
\[\alpha_f(N):=\bigcup \{\alpha(M) : M \subseteq N \text{ is \fg}\}.\]
\end{defn}

We give a lemma that may be well-known, but that we include in the interest of keeping this paper self-contained.

\begin{lemma}\label{lem:altfg}
Let $\cl$ be a functorial, residual closure operation.  Let $L \subseteq M$ be modules such that $A^{\cl}_B$ is defined whenever $T \subseteq M$ and $A \subseteq B \subseteq M/T$.  Then \[
L^{\cl_f}_M = \bigcup \{(L \cap U)^{\cl}_U : U \text{ is a \fg\ submodule of } M\}.
\]
\end{lemma}

\begin{proof}
For the forward containment, let $z \in L^{\cl_f}_M$.  Then there is some module $N$ with $L \subseteq N \subseteq M$, $N/L$ \fg\, and $z \in L^{\cl}_N$.  Set $x_1 := z$, and choose $x_2, \ldots, x_t \in N$ such that the images of the $x_j$ generate $N/L$.  That is, $L + U= N$, where $U := \sum_{j=1}^t R x_j$.  By functoriality, this means $\bar z \in 0^{\cl}_{N/L} = 0^{\cl}_{(L+U)/L} \cong 0^{\cl}_{U / (L \cap U)}$, by the Second Isomorphism Theorem.  Under this canonical isomorphism, the image of $z$ in $N/L$ maps to the image of $z$ in $U/(L \cap U)$.  So by residuality, we then have $z \in (L \cap U)^{\cl}_U$.

For the reverse containment, suppose $U$ is a \fg\ submodule of $M$ and $z \in (L \cap U)^{\cl}_U$.  Then $\bar z \in 0^{\cl}_{U / (L \cap U)} \cong 0^{\cl}_{U+L / L}$, so by the argument in the first paragraph, we have $z \in L^{\cl}_{U+L}$.  Additionally, $U+L$ is a submodule of $M$ containing $L$ such that $U+L / L \cong U / (U \cap L)$ is \fg\ (as it is isomorphic a quotient of the \fg\ module $U$). Hence $z\in L^{\cl_f}_M$.
\end{proof}

The following Theorem is analogous to (and for the most part a generalization of) \cite[Theorem 5.5]{nmeRG-cidual}.

\begin{thm}\label{thm:finintideal}
Let $(R,\m,k)$ be a complete Noetherian local ring.  Let $\alpha$ be a functorial submodule selector (i.e. a preradical) on $\cA$, the class of Artinian $R$-modules.  Let $\alpha_f$ be its finitistic version.  Let $\cl$ (resp. $\cl_f$) be the operation such that $L^{\cl}_M / L = \alpha(M/L)$ (resp. $L^{\cl_f}_M/ L = \alpha_f(M/L)$), where $M/L$ is Artinian.  Assume that $\cl$ and $\cl_f$ are closure operations.  Let $I$ be any ideal of $R$.  Then: \begin{align*}
    \alpha^\dual(I) &= \ann_R((\ann_E(I))^{\cl}_E) = \bigcap_{M \in \cA} \ann_R((\ann_M(I))^{\cl}_M) \\
    &\subseteq (\alpha_f)^\dual(I) = \ann_R((\ann_E(I))^{\cl_f}_E) = \bigcap_{M \subseteq E \text{ f.g.}} \ann_R((\ann_M(I))^{\cl}_M) 
    \\&=
    \bigcap_{\lambda(M) < \infty} \ann_R((\ann_M(I))^{\cl}_M)
    \subseteq \bigcap_{\lambda(R/J)<\infty} \ann_R((\ann_{R/J}(I))^{\cl}_{R/J}) \\ &= \bigcap_{\lambda(R/J)<\infty} (J:(J:I)^{\cl}_R).
\end{align*}
Moreover, the last containment is an equality when $R$ is approximately Gorenstein.  In that case if $\{J_t\}$ is a decreasing nested sequence of irreducible ideals that is cofinal with the powers of $\m$, then in fact we have \[
(\alpha_f)^\dual(I) = \bigcap_{t\geq 0} (J_t : (J_t : I)^{\cl}_R).
\]
\end{thm}

\begin{rem}For examples of computations using the last part of the result, see Section \ref{sec:examples}.  For an explanation of the term ``Artinistic'' to describe the phenomenon here, see the end of Section~\ref{sec:expansions}, in particular Definition \ref{def:cofin} and Theorem \ref{thm:cogen}.
\end{rem}
 
\begin{proof}
For the first equality, we have \begin{align*}
\alpha^\dual(I) &= \left(\frac{I^\vee}{\alpha(I^\vee)}\right)^\vee = \left(\frac{E/\ann_EI}{(\ann_E(I)^{\cl}_E / \ann_E(I))}\right) \\
&= \left(\frac E {\ann_E(I)^{\cl}_E}\right)^\vee = \ann_R(\ann_E(I)^{\cl}_E).\end{align*}
Note that this also establishes the third equality.

For the second equality, we prove both containments.  For the forward containment, let $M \in \cA$.  Then there is some positive integer $t$ and some inclusion map $j: M \into E^{\oplus t}$.  Then $j(\ann_M(I)) \subseteq \ann_{E^{\oplus t}}(I)$.  Thus since $\cl$ is functorial (by the assumption on $\alpha$), we have \[
j(\ann_M(I)^{\cl}_M) \subseteq
\ann_{E^{\oplus t}}(I)^{\cl}_{E^{\oplus t}} = ((\ann_E(I))^{\cl}_E)^{\oplus t}.
\]
Hence, \begin{align*}
\ann_R(\ann_M(I)^{\cl}_M) &= \ann_R(j(\ann_M(I)^{\cl}_M)) \\
&\supseteq \ann_R((\ann_E(I)^{\cl}_E)^{\oplus t}) = \ann_R(\ann_E(I)^{\cl}_E).
\end{align*}
For the reverse containment, we merely note that $E \in \cA$.

The first displayed containment holds because $\alpha_f \leq \alpha$ and $\dual$ is order-reversing.

For the fourth equality, we prove both containments.  
For the forward containment, let $M$ be a \fg\ submodule of $E$.  
Then by Lemma~\ref{lem:altfg}, 
\[(\ann_MI)^{\cl}_M = ((\ann_EI) \cap M)^{\cl}_M \subseteq (\ann_EI)^{\cl_f}_E,\] 
so $\ann_R((\ann_EI)^{\cl_f}_E) \subseteq \ann_R((\ann_MI)^{\cl}_M)$ for all such $M$.  For the reverse containment, let $a\in \displaystyle \bigcap_{M \subseteq E \text{ f.g.}} \ann_R((\ann_M(I))^{\cl}_M)$ and let $z \in (\ann_EI)^{\cl_f}_E$.  Then by Lemma~\ref{lem:altfg}, there is some \fg\ submodule $M \subseteq E$ with $z\in (M \cap \ann_EI)^{\cl}_M = (\ann_MI)^{\cl}_M$. Hence $az=0$.

For the fifth equality, we prove both containments.  For the forward containment, let $N$ be an $R$-module of finite length.  Then we have $N = \bigoplus_{j=1}^t N_j$, where each $N_j$ is indecomposable and of finite length.  Hence, for each $j$ there is some injective $R$-linear map $i_j: N_j \hookrightarrow E$.  Set $M_j := i_j(N_j)$.  Then each $M_j$ is a \fg\ submodule of $E$.  Thus we have \begin{align*}
\ann_R((\ann_NI)^{\cl}_N) &= \ann_R\left(\bigoplus_{j=1}^t (\ann_{N_j}I)^{\cl}_{N_j}\right) = \bigcap_{j=1}^t \ann_R((\ann_{N_j}I)^{\cl}_{N_j}) \\
&= \bigcap_{j=1}^t \ann_R((\ann_{M_j}I)^{\cl}_{M_j}) \supseteq \bigcap_{M \subseteq E \text{ f.g.}}  \ann_R((\ann_M(I))^{\cl}_M).
\end{align*}
For the reverse containment, we merely recall that any \fg\ submodule of $E$ has finite length. 

For the sixth equality, let $J$ be an arbitrary ideal of finite colength.  Then $\ann_{R/J}I = (J:I) / J$, so by residuality of the closure operation, we have $(\ann_{R/J}I)^{\cl}_{R/J} = (J:I)^{\cl}_R / J$.  Now, for any ideal $K$ of $R$ that contains $J$, we have $\ann_R(K/J) = (J:K)$.  Hence in particular, we have \[
\ann_R((\ann_{R/J}(I))^{\cl}_{R/J}) = \ann_R((J:I)^{\cl}_R/J) = (J:(J:I)^{\cl}_R).
\]

Now consider the case where $R$ is approximately Gorenstein, with the ideals $J_t$ as given in the statement of the theorem.  Let $N_t := \ann_E J_t$ for each $t\in \N$.  By the theory of approximately Gorenstein rings \cite[Page 157]{HoFNDTC}, we have $N_t \cong R/J_t$ as $R$-modules, $N_t \subseteq N_{t+1}$ for each $t$, and $E = \bigcup_{t\in \N} N_t$.  Now let $M$ be a \fg\ submodule of $E$.  Say $M = \sum_{i=1}^m R z_i$.  Then each $z_i \in N_{t_i}$ for some $i$.  Let $s = \max \{ t_i \mid 1\leq i \leq m\}$. Then every $z_i \in N_s$, whence $M \subseteq N_s$. 
Since $\cl$ preserves containment in submodules and ambient modules, we have \[
(\ann_MI)^{\cl}_M \subseteq (\ann_MI)^{\cl}_{N_s} \subseteq (\ann_{N_s}I)^{\cl}_{N_s}.
\]
Hence, \begin{align*}
\bigcap_{t \geq 0} (J_t : (J_t :I)^{\cl}_R) &\subseteq (J_s : (J_s:I)^{\cl}_R) = \ann_R((\ann_{R/J_s}I)^{\cl}_{R/J_s}) \\
&= \ann_R((\ann_{N_s}I)^{\cl}_{N_s}) \subseteq \ann_R((\ann_MI)^{\cl}_M).
\end{align*}
But the intersection of all such ideals has already been shown to be equal to $(\alpha_f)^\dual(I)$.

We have shown that $\bigcap_t (J_t : (J_t :I)^c_R) \subseteq (\alpha_f)^\dual(I)$.  But the left hand side contains $\bigcap_{\lambda(R/J)<\infty} (J : (J:I)^c_R)$ since each $J_t$ has finite colength. This shows that the last containment of the first part of the theorem is an equality for approximately Gorenstein rings.
\end{proof}

Note that the above result already has consequences beyond that of  \cite[Theorem 5.5]{nmeRG-cidual}, even in the case where $I=R$.

\begin{cor}\label{cor:irredseq}
Let $(R,\m,k)$ be a complete approximately Gorenstein Noetherian local ring.  Let $\cl$ be a functorial and residual closure operation.  Let $\{J_t\}_{t\in \N}$ be a sequence of irreducible $\m$-primary ideals cofinal with the powers of $\m$ and $\cl$-closed.  Then for every ideal $I$ of $R$, we have $I^\cl = I$, and for every finitely generated $R$-module $M$, we have $0^\cl_M = 0$.
\end{cor}

\begin{proof}
Let $\alpha$ be the preradical associated to $\cl$, and $\alpha_f$ the preradical associated to $\cl_f$.  Then by Theorem~\ref{thm:finintideal}, we have \begin{align*}
(\alpha_f)^\dual(R) &= \bigcap_t (J_t : J_t^\cl) = \bigcap_t (J_t : J_t) = R \\
&= \bigcap_{\len(M) < \infty} \ann_R(0^\cl_M) = \bigcap_{\len(R/J) < \infty} (J : J^\cl).
\end{align*}
Hence, $0$ is $\cl$-closed in every finite-length module (since $1\in R = $ the common annihilator of their closures), and similarly every finite colength ideal is $\cl$-closed.

Now let $M$ be an arbitrary finitely generated $R$-module.  Then for any positive integer $s$, we have that $M / \m^s M$ is finite length.  Hence by the residual property, we have \[
(\m^s M)^\cl_M / \m^s M = 0^\cl_{M / \m^s M} = 0.
\]
That is, $\m^s M$ is a $\cl$-closed submodule of $M$.  But by the Krull intersection theorem, $0 = \bigcap_{s\in \N} \m^s M$.  Hence $0$, being an intersection of $\cl$-closed submodules of $M$, is itself $\cl$-closed (cf. \cite[Proposition 2.1.3]{nme-guide2}, where the result is stated for ideals but whose proof extends immediately to modules).

Finally, let $I$ be an arbitrary ideal.  Then $R/I$ is a finitely generated $R$-module, so by the above, we have $0 = 0^\cl_{R/I} = I^\cl / I$, whence $I^\cl = I$.
\end{proof}

The above is especially interesting in the Gorenstein case, so we state it separately as follows.

\begin{cor}\label{cor:Gorsop}
Let $(R,\m,k)$ be a complete Gorenstein Noetherian local ring.  Let $\cl$ be a functorial and residual closure operation.  Suppose there is some system of parameters ${\mathbf x} := x_1, \ldots, x_d$ such that the ideals $J_t := (x_1^t, \ldots, x_d^t)$ are $\cl$-closed for infinitely many $t\in \N$.  Then for every ideal $I$ of $R$, we have $I^\cl = I$, and for every finitely generated $R$-module $M$, we have $0^\cl_M = 0$.
\end{cor}

\begin{proof}
Recall \cite[Theorem 18.1]{Mats} that in this case, any ideal generated by a full system of parameters is irreducible.  Hence some subsequence of the ideals $J_t$ satisfies the conditions of Corollary~\ref{cor:irredseq}.
\end{proof}

Hence, we obtain a unified proof of the following results, which previously seemed to require completely different proofs:

\begin{cor}\label{cor:Fproperties}
Let $(R,\m,k)$ be a complete Gorenstein Noetherian local ring.  \begin{enumerate}
    \item\ \cite[Remark 3.8, in char $p$ only]{EneHo-structure} If $R$ is equicharacteristic and $F$-rational, then it is $F$-regular.
    \item\ {\cite[Lemma 3.3]{Fe-Fpure}} If $R$ is of prime characteristic $p>0$ and $F$-injective, then it is $F$-pure.
\end{enumerate}
\end{cor}

\begin{proof}
For the first item, note that the definition of $F$-rational is that ideals generated by systems of parameters are tightly closed.
For the second item, 
if a ring is Cohen-Macaulay, it is $F$-injective if and only if any ideal generated by a system of parameters is Frobenius closed \cite[Remark 1.9]{FeWa-Freg}.

Both results now follow from Corollary \ref{cor:Gorsop}.
\end{proof}

It happens that the assumption of Gorensteinness is crucial in Corollary~\ref{cor:Gorsop}.  To see this, we look at the special case where $R$ is Cohen-Macaulay, and the closure operation in question is $\cl_{\omega_R}$.
In that case, \emph{every} ideal generated by a system of parameters is $\cl_{\omega_R}$-closed, yet $\cl_{\omega_R}$ is \emph{never} trivial unless $R$ is Gorenstein:

\begin{lemma}
Let $R$ be a complete Cohen-Macaulay Noetherian local ring, let $\omega_R$ be its canonical modules, and let $\cl = \cl_{\omega_R}$.  Then for any ideal $I$ generated by part of a system of parameters, we have $I^\cl=I$.
\end{lemma}

\begin{proof}
Let $\bfx$ be a full system of parameters.  Let $J=(\bfx)$.  Then by \cite[Theorem 3.3.4(a)]{BH}, $\omega_R / J \omega_R \cong E_{R/J}(k)$, which by \cite[Proposition 3.2.12]{BH} is a faithful $(R/J)$-module. Let $a\in J^\cl = \ann_R(\omega_R / J \omega_R)$.  Then in $R/J$, we have $\bar{a} \in \ann_{R/J}(\omega_R / J \omega_R) = \ann_{R/J}(E_{R/J}(k))= \bar{0}$ in $R/J$.  Hence, $a\in J$.
 
Now take an arbitrary parameter ideal $I = (x_1,\ldots, x_t)$.  Complete to a full system of parameters $x_1, \ldots, x_t, x_{t+1}, \ldots, x_d$.  Then for every positive integer $s$, we have that the ideal $J_s := I + (x_{t+1}^s, \ldots, x_d^s)$ is generated by a full system of parameters.  But $I = \bigcap_s J_s$, and every $J_s$ is $\cl$-closed by the first paragraph of the proof. Hence $I$, being an intersection of $\cl$-closed ideals, is itself $\cl$-closed.
\end{proof}

On the other hand, by \cite[Lemma 2.1]{HHS-tracecan} and the proof of \cite[Corollary 3.18]{PeRG}, we have that that $\cl_{\omega_R}$ is trivial if and only if $R$ is Gorenstein, even though by the above, this closure is always trivial on \emph{parameter} ideals.

\section{Many module and algebra closures are Nakayama}\label{sec:manyNAK}

In this section we prove that certain module and algebra closures are Nakayama closures, so that the results of this paper apply to them.

\begin{thm}\label{thm:fgNak}
Let $(R,\m) \rightarrow (S,\n)$ be a local homomorphism of Noetherian local rings, and let $B$ be a finitely generated $S$-module.  Let $\cM$ be the class of finitely generated $R$-modules, and consider the closure operation $\cl := \cl_B$ on $\cM$.  Then $\cl$ is a Nakayama closure.

In particular, this holds when $B$ is either a finitely generated $R$-module (i.e. the case $R=S$) or any Noetherian local $R$-algebra (i.e. the case $S=B$).
\end{thm}

\begin{proof}
Since $\cl$ is residual, it is enough to show that for any $L \subseteq M \in \cM$, if $L \subseteq (\m L)^\cl_M$, we have $L \subseteq 0^\cl_M$.  Accordingly suppose $L \subseteq (\m L)^\cl_M$.  Let $H$ be the image of the induced map $L \otimes_R B \rightarrow M \otimes_R B$.  By the assumption, $H$ is in fact contained in the image of $\m L \otimes_R B \rightarrow M \otimes_R B$.  But any element of the latter is of the form \[
\sum_{j=1}^t (m_j x_j) \otimes b_j = \sum_{j=1}^t m_j (x_j \otimes b_j),
\]
with $m_j \in \m$, $x_j \in L$, and $b_j \in B$.  But the latter representation of such an element is clearly in $\m H$, since each $x_j \otimes b_j$ is in $H$.  This in turn is contained in $\n H$ since $\m S \subseteq \n$ and $H$ is an $S$-module.

Hence, $H \subseteq \n H$.  But $H$ is a submodule of $M \otimes_R B$, which is a finitely generated $S$-module.  Since $S$ is Noetherian, it follows that $H$ is itself a finitely generated  $S$-module.  But then the Nakayama lemma (applied to $H$ as a finitely generated $S$-module) implies that $H=0$.  In other words, $L \subseteq 0_M^\cl$.
\end{proof}

To go further, we recall the direct limit scaffolding from \cite{nmeRG-cidual}. Let $\Gamma$ be a directed set.  Let $\{\alpha_i \mid i \in \Gamma\}$ be a directed system of submodule selectors. Then $\alpha := \varinjlim \alpha_i$, defined by $\alpha(M) := \bigcup_{i \in \Gamma} \alpha_i(M)$ \cite[Definition 7.1]{nmeRG-cidual}, is a submodule selector as well.  Moreover, \cite[Proposition 7.2]{nmeRG-cidual} if each of the $\alpha_i$ arises from a functorial residual closure operation, then so does $\alpha$.

In particular, we have the following for algebra closures, which is implicitly used for example in \cite{RG-bCMsing} and \cite{PeRG}, but not explicitly stated there.

\begin{lemma}
Let $\{A_i\}_{i \in I}$ be a direct limit system of $R$-algebras with $R$-algebra homomorphisms, with direct limit $A$. Then $\sum_{i \in I} \cl_{A_i}$ is a closure operation and in fact is equal to $\cl_A$.
\end{lemma}

\begin{rem}
Using the notation of \cite{nmeRG-cidual}, we may also refer to this closure as $\varinjlim \cl_{A_i}$.
\end{rem}

\begin{proof}
We check equality, which is enough to show that $\cl:=\sum_{i \in I} \cl_{A_i}$ is a closure operation. First, note that each $A_i$ has an $R$-algebra map to $A$, so $\cl_{A_i} \le \cl_A$ for all $i \in I$ \cite[Proposition 3.6]{RG-bCMsing}. This gives the containment $\cl \le \cl_A$. For the other containment, let $M$ be an $R$-module, and $u \in 0_M^{\cl_A}$. This means that $1 \otimes u =0$ in $A \otimes_R M$. Since $A$ is the direct limit of the $A_i$, there must exist $i \in I$ such that $1 \otimes u=0$ in $A_i \otimes_R M$, which implies that $u \in 0_M^{\cl_{A_I}} \subseteq 0_M^{\cl}$. Since all of the closures in question are residual, this proves the result.
\end{proof}

\begin{prop}\label{prop:Naklimit}
Let $(R,\m)$ be a Noetherian local ring. Let $\Gamma$ be a directed poset. Let $\{\cl_j\}_{j \in \Gamma}$ be a directed set of residual functorial Nakayama closure operations (that is, $\Gamma$ is a directed set, and whenever $j \leq j'$, we have $\cl_j \leq \cl_{j'}$). Then $\cl := \displaystyle \lim_{\overset{\rightarrow}{j \in \Gamma}} \cl_j$, provided it is idempotent (and hence a closure operation), is also residual, functorial, and Nakayama.
\end{prop}

\begin{proof}
Residual and functorial follow from \cite[Proposition 7.2]{nmeRG-cidual}.

To see the Nakayama property, let $L \subseteq M$ be finitely generated $R$-modules, and assume $L \subseteq (\m L)_M^\cl$.  Let $z_1, \ldots, z_t$ be a generating set for $L$.  Then there exist $j_1, \ldots, j_t \in \Gamma$ with $z_i \in (\m L)_M^{\cl_{j_i}}$ for each $1\leq i \leq t$. Choose $j \in \Gamma$ with $j \geq j_i$ for $1\leq i \leq t$ (which exists since $\Gamma$ is a directed set).  Then each $z_i \in (\m L)_M^{\cl_j}$, whence we have $L \subseteq (\m L)_M^{\cl_j}$ since the $z_i$ generate $L$.  But then since $\cl_j$ is a Nakayama closure, we have $L \subseteq 0_M^{\cl_j} \subseteq 0_M^\cl$.
\end{proof}

Next we need a lemma that may be well known, but we don't know a reference so we include it and its proof for the reader's convenience.

\begin{lemma}\label{lem:dirlim}
Let $(R,\m)$ and $(B,\n)$ be (not necessarily Noetherian) local rings, and let $\phi: R \rightarrow B$ be a local homomorphism.  Then as an $R$-algebra, $B$ is the direct limit, via a directed indexing set $I$, of local $R$-algebras $B_i$ that are essentially of finite type over $R$, such that $\m B_i \neq B_i$ for each $i \in I$.
\end{lemma}

\begin{proof}
Let $G$ be the collection of all finite subsets of $B$.  Let $\Gamma$ be an indexing set in bijective correspondence with $G$.  For any $X \in G$, we write $X=X_i$ where $i\in \Gamma$ is the index corresponding to $X$.  We partially order $\Gamma$ such that $i \leq j$ if and only if $X_i \subseteq X_j$.  For each $i\in \Gamma$, set $A_i := \phi(R)[X_i]$,  i.e. the subring of $B$ generated by $\phi(R)$ and $X_i$. 
Then it is clear that if $i \leq j$, we have a natural corresponding $R$-algebra map $\mu_{ij}: A_i \ra A_j$, given by simple inclusion.  Now, for each $i\in \Gamma$, set $\n_i := A_i \cap \n$.  Let $B_i := (A_i)_{\n_i}$, and whenever $i\leq j$ define $\nu_{ij}: B_i \ra B_j$ by $\nu_{ij}(a/s) := a/s$.

To see that $\nu_{ij}$ is well-defined, note first that if $s\in A_i \setminus \n_i$, we have $s \notin \n$, whence $s \in A_j \setminus \n = A_j \setminus \n_j$.  Moreover, if $a/s = b/t$ in $B_i$, then there is some $u \in A_i \setminus \n$ with $uta=usb$.  But we have $u,t,s \in A_j \setminus \n_j$, whence $a/s=b/t$ in $B_j$.

Now for each $i$, define $\nu_i: B_i \rightarrow B$ in the same fashion.  That is, $\nu_i(a/s) = a/s$, which is well-defined for the same reasons as above.  Then it is elementary that the $B_i$, along with the maps $\nu_{ij}$ and $\nu_i$, form a direct limit system with direct limit of $B$.
\end{proof}

\begin{cor}\label{cor:localalgNak}
Let $(R,\m)$ be a Noetherian local ring, and let $(B,\n)$ be a local $R$-algebra such that $\m B \neq B$.  Then $\cl_B$ is a Nakayama closure on finitely generated $R$-modules.
\end{cor}

\begin{proof}
Construct the direct limit system of $B_i$ as in Lemma~\ref{lem:dirlim}.  Since each $B_i$ is essentially of finite type over the Noetherian ring $R$, it is itself Noetherian.  Moreover, by construction the homomorphisms $(R,\m) \ra (B_i, \n_i)$ are local.  Then by Theorem~\ref{thm:fgNak}, each $\cl_{B_i}$ is a Nakayama closure over $R$.  By Proposition~\ref{prop:Naklimit}, $\cl_B$ is then Nakayama as well.
\end{proof}

\begin{thm}\label{thm:algclosuresareNak}
Let $(R,\m)$ be a Noetherian local ring.  Let $B$ be an $R$-algebra such that $\m B$ is contained in the Jacobson radical of $B$.  Then $\cl_B$ is a Nakayama closure on finitely generated $R$-modules.
\end{thm}

\begin{proof}
Let $L \subseteq M$ be finitely generated $R$-modules.  Suppose $L \subseteq (\m L)^{\cl_B}_M$.  Then for any maximal ideal $P$ of $B$, we have $\cl_B \leq \cl_{B_P}$ (since $B_P$ is a $B$-algebra), so $L \subseteq (\m L)^{\cl_{B_P}}_M$.  Since $\cl_{B_P}$ is Nakayama by Corollary~\ref{cor:localalgNak}, we have $L \subseteq 0^{\cl_{B_P}}_M$.  That is, for all maximal ideals $P$ of $B$, we have, for all $z\in L$, that $z \otimes 1 = 0$ in $M \otimes_R B_P = (M \otimes_R B)_P$.  Since vanishing is a local property, it follows that $z\otimes 1 = 0$ in $M \otimes_RB$ for all $z\in L$.  That is, $L \subseteq 0^{\cl_B}_M$.
\end{proof}

\begin{cor}\label{cor:F+bigNak}
Let $(R,m)$ be a Noetherian local ring.
Then Frobenius closure (in characteristic $p>0$) and plus closure (when $R$ is complete) are Nakayama closures, as is $\cl_B$ for any big \CM\ $R$-algebra $B$ that is local.
\end{cor}

\begin{proof}
Note that Frobenius closure is $\cl_{R^{1/p^\infty}}$ and plus closure is $\cl_{R^+}$. Each of the algebras $R^{1/p^\infty}$ and $R^+$ is local, with maximal ideal containing $\m$, so we can apply Theorem \ref{thm:algclosuresareNak}.
\end{proof}

\section{Nakayama interiors are dual to Nakayama closures}
\label{sec:nakayama}

In this section, we define a Nakayama interior and prove that the dual of a Nakayama closure is a Nakayama interior. $R$ will be a commutative ring with additional hypotheses as specified.

\begin{defn}
\label{def:Nakayamainterior}
Let $\int$ be an interior operation on the class of Artinian $R$-modules, where $(R,\m)$ is a Noetherian local ring. We say that $\int$ is \emph{Nakayama} if whenever $A \subseteq B$ are Artinian modules such that $\int(A :_B \m) \subseteq A$, we have $\int(A)=\int(B)$.
\end{defn}

We need the following presumably well-known fact:
\begin{lemma}\label{lem:simple}
If $R$ is any associative (not necessarily commutative) ring and $A$ is a nonzero Artinian left $R$-module, then $A$ contains a simple left $R$-module.
\end{lemma}   

\begin{proof}
If $A$ doesn't contain a simple module, then suppose there is a proper descending chain of length $n$ in $A$:  $A_0\supsetneq A_1 \supsetneq\cdots\supsetneq A_n \neq 0$. Since $A_n$ is nonzero and not simple,  there is an $A_{n+1} \neq 0$ properly contained in $A_n$.  Hence there is an infinite descending chain which is a contradiction since $A$ is Artinian.
\end{proof}

This allows us to give an example of a Nakayama interior:

\begin{lemma}   
\label{lem:identityisNak}
Let $(R,\m)$ be a Noetherian local ring.  Then the identity operation is a Nakayama interior on the class of Artinian modules.
\end{lemma}

\begin{proof}
Let $A \subseteq B$ be Artinian modules such that $(A :_B \m) \subseteq A$.  If $A \neq B$, then $B/A$ is a nonzero Artinian module, so it contains a nonzero simple submodule $S$.  Let $0 \neq x \in S$. Then $x = y+A \in B/A$ for some $y\in B$, and since $S \cong R/\m$ as $R$-modules, we have $\m x = 0$, which means that $\m y \subseteq A$.  That is, $y \in (A :_B \m) \setminus A$, which contradicts the assumption.  Hence $A=B$.
\end{proof}

As promised, we prove that the dual of a Nakayama closure is a Nakayama interior. First we prove a lemma.

\begin{lemma}\label{lem:NmodIL}
Let $R$ be a complete Noetherian local ring.  Let $L \subseteq M$ be Matlis-dualizable $R$-modules (i.e. $R$-modules isomorphic to their double Matlis duals), let $B := M^\vee$, and let $A \subseteq B$ be the $R$-submodule of $B$ such that $A = (M/L)^\vee = \{f \in M^\vee \mid L \subseteq \ker f\}$.  Let $I$ be an ideal of $R$.  Then when thought of as a submodule of $M^\vee$, we have $(M/I L)^\vee = (A :_B I)$.
\end{lemma}

\begin{proof}
Let $f \in M^\vee = \Hom_R(M,E)$.  We need to show that $I L \subseteq \ker f$ if and only if for all $g\in I f$, we have $L \subseteq \ker g$.

Accordingly, suppose that $I L \subseteq \ker f$.  Let $g \in I f$.  Then $g=\mu f$ for some $\mu\in I$.  Then for any $z \in L$, we have $g(z) = (\mu f)(z) = \mu \cdot f(z) = f(\mu z) = 0$ since $\mu z \in IL$.  Hence $L \subseteq \ker g$.

Conversely, suppose that $L \subseteq \ker g$ for all $g \in If$.  Let $z \in I L$.  Then we have $z = \sum_{j=1}^t \mu_j y_j$, where $\mu_j \in I$ and $y_j \in L$.  Let $g_j = \mu_j \cdot f$ for $1\leq j \leq t$.  Note that $g_j \in If$, so that $L \subseteq \ker g_j$ for all $1\leq j \leq t$.  Then we have \[
f(z) = f(\sum_j \mu_j y_j) = \sum_j \mu_j \cdot f(y_j) = \sum_j (\mu_j f)(y_j) = \sum_j g_j(y_j) = 0.
\]
Hence $IL \subseteq \ker f$.
\end{proof}

\begin{prop}
\label{prop:nakayamaclosureint}
Let $(R,\m)$ be a complete Noetherian local ring.  Let $\cl$ be a residual closure operation on the category of \fg\ $R$-modules, and let $\int$ be the interior operation on the category of Artinian $R$-modules given by $\int(A) = \left( \frac{A^\vee}{0^\cl_{A^\vee}}\right)^\vee$, i.e. the interior operation dual to $\cl$.  Then $\cl$ is a Nakayama closure if and only if $\int$ is a Nakayama interior.
\end{prop}

\begin{proof}
First assume $\cl$ is a Nakayama closure operation.  Let $A \subseteq B$ be Artinian $R$-modules such that $\int(A :_B \m) \subseteq A$.  Let $M = B^\vee$, and set 
\[L := \{f\in B^\vee \mid A \subseteq \ker f\} \cong (B/A)^\vee.\]  Then by Lemma~\ref{lem:NmodIL}, $(A :_B \m) = (M/\m L)^\vee$.  The fact  that $\int(A :_B\m) \subseteq A$ means then that \begin{align*}
(M/L)^\vee &= A \supseteq \int(A :_B\m) = \int((M/\m L)^\vee) \\
&\cong \left(\frac{M/\m L} {0^\cl_{M/\m L}}\right)^\vee = (M / (\m L)^\cl_M)^\vee.
\end{align*}
  Applying Matlis duality, we obtain $M/L \twoheadrightarrow M / (\m L)^\cl_M$, which is to say that $L \subseteq (\m L)^\cl_M$.  Then by the Nakayama closure property, we obtain $L \subseteq 0^\cl_M$.  This then implies $M/L \twoheadrightarrow M / 0^\cl_M$.  Applying Matlis duality, we have $\int(B) = (M / 0^\cl_M)^\vee \subseteq (M/L)^\vee = A$, as was to be shown.

Conversely, assume $\int$ is a Nakayama interior operation, and let $L \subseteq N \subseteq M$ be \fg\ $R$-modules such that $N \subseteq (L + \m N)^\cl_M$.  Without loss of generality (by the residual property of $\cl$), we may assume $L=0$.  Now let $B = M^\vee$ and $A = (M/N)^\vee \subseteq B$. The fact that $N \subseteq (\m N)^\cl_M$ means that $M/N \onto M / (\m N)^\cl_M$, which translates to \begin{align*}
\int(A :_B \m) &= \int((M/\m N)^\vee)=\left(\frac{M/\m N} {0^\cl_{M/\m N}}\right)^\vee\\
 &= (M/(\m N)^\cl_M)^\vee \subseteq (M/N)^\vee = A. 
\end{align*}
 Then since $\int$ is a Nakayama interior, it follows that \[
\left(\frac M {0^\cl_M}\right)^\vee = \left(\frac{B^\vee}{0^\cl_{B^\vee}}\right)^\vee = \int(B) \subseteq A = (M/N)^\vee.
\]
  Thus, we have $M/N \onto M / 0^\cl_M$, whence $N \subseteq 0^\cl_M$.
\end{proof}

\section{i-Expansions, co-generating sets, and core-hull duality}
\label{sec:expansions}

We define i-expansions, 
discuss the duality between cl-reductions and i-expansions, define the i-hull, and incorporate the notion of co-generation in order to prove that the i-hull is dual to the cl-core.   We will close the section by applying these tools to define the Artinistic version of a submodule selector; in turn, we use the definition to obtain a new characterization of the finitistic (i.e. classical) test ideal.

\begin{defn}
Suppose $R$ is an associative (not necessarily commutative) ring and $A \subseteq B$ are left $R$-modules.  Suppose $\int$ is an interior operation that operates on at least the submodules of $B$ that contain $A$.  We say $C$ with $A\subseteq C \subseteq B$ is an \emph{$\int$-expansion of $A$ in $B$} if $\int(A)=\int(C)$.
\end{defn}

\textbf{Setup:}
If $\int$ is an interior operation on Artinian $R$-modules, set $\alpha=\int^{\smile}$, and $\cl$ to be the corresponding residual closure operation.

\begin{lemma}
\label{lemma:inclusion}
Let $R$ be a complete Noetherian local ring. Let $C \subseteq B$ be $R$-modules, and $j:C \hookrightarrow B$ the inclusion. Let $\pi=j^\vee:B^\vee \to C^\vee$ and $L \subseteq C$. Then $\pi^{-1}\left((C/L)^\vee\right)=(B/L)^\vee.$
\end{lemma}

\begin{proof}
Let $g \in \pi^{-1}\left((C/L)^\vee\right),$ so that $g \in \Hom_R(B,E)$. Then \[\pi(g) \in (C/L)^\vee \subseteq C^\vee,\] i.e., $g \circ j \in (C/L)^\vee$. This implies that $g \circ j$ kills $L$. Hence $L \subseteq \ker(g \circ j),$ so $L=L \cap C \subseteq \ker(g).$ This implies that $g \in (B/L)^\vee$.

Now suppose that $g \in (B/L)^\vee$. Then $g \in \Hom_R(B/L,E)$, so $L \subseteq \ker(g)$. Then $\pi(g)=g \circ j:C \to E$ must also kill $L$, so $g \in \pi^{-1}\left((C/L)^\vee\right)$.
\end{proof}

\begin{thm}
\label{thm:expansiondualisreduction}
Let $R$ be a Noetherian local ring and $A \subseteq B$ Matlis-dualizable $R$-modules.  Let $\int$ be an interior operation, let $\alpha := \int^\dual$, and let $\cl$ be the corresponding residual closure operation.  There exists an order reversing one-to-one correspondence between the poset of $\int$-expansions of $A$ in $B$ and the poset of $\cl$-reductions of $(B/A)^{\vee}$ in $B^{\vee}$. Under this correspondence, an $\int$-expansion $C$ of $A$ in $B$ maps to $(B/C)^\vee$, a $\cl$-reduction of $(B/A)^\vee$ in $B^\vee$.
\end{thm}

\begin{proof}
Let $C$ be an $\int$-expansion of $A$ in $B$; in other words, $\int(C) \subseteq A\subseteq C$.
Let $\pi:B^\vee \to B^\vee/(B/C)^\vee=C^\vee$ be the quotient map. Let $\alpha := \int^\dual$, let $\cl$ be the corresponding residual closure operation. We have:

\begin{align*}
\left((B/C)^\vee\right)^\cl_{B^\vee} &= \pi^{-1}\left(\alpha(B^\vee/(B/C)^\vee)\right) \\ &= \pi^{-1}(\alpha(C^\vee)) \\
&= \pi^{-1}\left((C/\int(C))^\vee\right),
\end{align*}
where the last equality follows because $\alpha=\int^\dual$. By Lemma \ref{lemma:inclusion}, we have
\[
    \pi^{-1}\left((C/\int(C))^\vee\right) 
    =\left(B/\int(C)\right)^\vee.
\]
Since $\int(C)=\int(A)$, this is equal to
\[\left(B/\int(A)\right)^\vee,\]
and the latter is equal to $\left((B/A)^\vee\right)^\cl_{B^\vee}$ by the same argument used for $C$. Hence
\[\left((B/C)^\vee\right)^\cl_{B^\vee}=\left((B/A)^\vee\right)^\cl_{B^\vee}.\]
This implies that $(B/C)^\vee$ is a cl-reduction of $(B/A)^\vee$; in other words,  $(B/C)^\vee \subseteq (B/A)^\vee \subseteq ((B/C)^\vee)_{B^{\vee}}^{\cl}$, establishing one direction of the correspondence.

Now let $N := (B/A)^{\vee} \subseteq B^{\vee}=:M$ and let $L$ be a $\cl$-reduction of $N$ in $M$; in other words, $L \subseteq N \subseteq L^{\cl}_M$.
We need to show $(M/L)^\vee$ is an $\int$-expansion of $(M/N)^\vee$ in $M^\vee$.  
 Note that the natural surjections $M \onto M/L \onto M/N$ yield inclusions $(M/N)^\vee \hookrightarrow (M/L)^\vee \hookrightarrow M^\vee$, so that the above make sense.  Accordingly, since $L^\cl_M = N^\cl_M$, we have \begin{align*}
    \int((M/L)^\vee) &= \alpha^\dual((M/L)^\vee) = \left(\frac{M/L}{\alpha(M/L)}\right)^\vee = \left(\frac{M/L}{L^\cl_M / L}\right)^\vee \\
    &= (M/L^\cl_M)^\vee = (M/N^\cl_M)^\vee = \left(\frac{M/N}{\alpha(M/N)}\right)^\vee \\
    &= \alpha^\dual((M/N)^\vee) = \int((M/N)^\vee) \qedhere
    \end{align*}
\end{proof}

\begin{prop}
\label{prop:maxintexpansions}
Let $(R, \m, k)$ be a complete Noetherian local ring.
Let $A \subseteq B$ be Artinian $R$-modules and $\int$ a Nakayama interior defined on Artinian $R$-modules.  Maximal $\int$-expansions of $A$ exist in $B$.  In fact, if $C$ is an $\int$-expansion of $A$ in $B$, then there is some maximal expansion $D$ of $A$ in $B$ such that $A \subseteq C \subseteq D \subseteq B$.
\end{prop}

\begin{proof}
Recall that the dual statement (Proposition~\ref{NR Ext}) holds for minimal $\cl$-reductions and \fg\ $R$-modules.

By Theorem~\ref{thm:expansiondualisreduction}, $(B/C)^\vee$ is a $\cl$-reduction of $(B/A)^\vee$ in $B^\vee$.  Hence by Proposition \ref{NR Ext},
there is some minimal $\cl$-reduction $U$ of $(B/A)^\vee$ in $B^\vee$ such that $U \subseteq (B/C)^\vee$. Let $D := (B^\vee / U)^\vee$.  Then by Theorem~\ref{thm:expansiondualisreduction}, $D$ is an $\int$-expansion of $A$ in $B$, and clearly $C \subseteq D$.

Now suppose that $D'$ is an $\int$-expansion of $A$ in $B$ with $D \subseteq D'$.  By Theorem~\ref{thm:expansiondualisreduction}, we then have that $(B/D')^\vee$ is a $\cl$-reduction of $(B/A)^\vee$ in $B^\vee$, and we have $(B/D')^\vee \subseteq (B/D)^\vee = U$.  By minimality of $U$, we therefore have $(B/D')^\vee = U = (B/D)^\vee$, whence $D=D'$.  Thus, $D$ is maximal.
\end{proof}

We now show that Proposition~\ref{prop:maxintexpansions} also holds for Noetherian $R$-modules, and in far greater generality.

\begin{prop}
\label{prop:maxexpansionsNoeth}
Let $R$ be an associative (i.e. not necessarily commutative) ring with identity.  Let $L\subseteq M$ be (left) $R$-modules, and let $\int$ be an interior operation on submodules of $M$.   Let $U$ be an $\int$-expansion of $L$ in $M$.  Assume $M/U$ is Noetherian.  Then there is an $R$-module $N$ with $U \subseteq N \subseteq M$, such that $N$ is a maximal $\int$-expansion of $L$ in $M$.
\end{prop}

\begin{proof}
Let $\cS$ be the set of all $R$-modules $D$ such that $U \subseteq D \subseteq M$ and $D$ is an $\int$-expansion of $L$.  Since $U \in \cS$, we have $\cS \neq \emptyset$.  Since $M/U$ is Noetherian and $\cS$ corresponds to a nonempty collection of submodules of $M/U$, $\cS$ contains a maximal element $N$.  Moreover, if $N'$ is an $\int$-expansion of $L$ in $M$ such that $N \subseteq N'$, then in particular $U \subseteq N' \subseteq M$, so that $N' \in \cS$.  Then by maximality of $N$, we have $N=N'$.  Hence $N$ is a maximal $\int$-expansion of $L$ in $M$.
\end{proof}

Next we describe a dual to the notion of a generating set. This will enable us to dualize a property of minimal reductions, namely that if $L$ is a minimal reduction of $N$ in $M$, a minimal generating set for $L$ extends to a minimal generating set for $N$.

\begin{defn}\label{def:cog}
Let $R$ be a Noetherian local ring, $L$ an $R$-module, and $g_1,\ldots,g_t \in L^\vee$. We say that the \textit{quotient of $L$ co-generated by $g_1,\ldots,g_t$} is $L/\left(\bigcap_i \ker(g_i)\right)$.

We say that $L$ is \textit{co-generated by $g_1,\ldots,g_t$} if $\bigcap_i \ker(g_i)=0$.

We say that a co-generating set for $L$ is \textit{minimal} if it is irredundant, i.e., for all $1 \le j \le t$, $\bigcap_{i \ne j} \ker(g_i) \ne 0$.
\end{defn}

\begin{lemma}
\label{lemma:imageinsocle}
Let $(R,\m,k)$ be a Noetherian local ring and $E=E_R(k)$.  
Let $V$ be an $R$-module such that $\m V=0$ and $g:V \to E$ an $R$-linear map. Then $\im(g)$ is contained in the unique copy of $k$ in $E$.
\end{lemma}

\begin{proof}
We have $\m \cdot \im(g)=0,$ so $\im(g) \subseteq \soc E$. Since the socle of $E$ is equal to this copy of $k$ inside of $E$, we get the desired result.
\end{proof}

\begin{lemma}
\label{lemma:cogeneration}
Let $R$ be a Noetherian local ring, and $A$ an $R$-module.  Assume $A$ is Matlis-dualizable (e.g. if $A$ is Artinian). The elements $g_1,\ldots,g_t \in A^\vee$ co-generate $A$ if and only if they generate $A^\vee$.
\end{lemma}

\begin{proof}
Let $\psi: R^{\oplus t} \rightarrow A^\vee$ be the map given by the row matrix $[\begin{matrix} g_1 & g_2 & \cdots & g_t\end{matrix}]$.  This induces a dual map $\psi^\vee: A^{\vee \vee} \rightarrow E^{\oplus t}$.  Let $j: A \rightarrow A^{\vee \vee}$ be the biduality map, which is an isomorphism by the hypotheses of the lemma.  Then $\phi := \psi^\vee \circ j: A \rightarrow E^{\oplus t}$ is given by the column matrix $\left[\begin{matrix} g_1\\ g_2\\ \vdots \\ g_t \end{matrix}\right]$, sending each $a\in A$ to $(g_1(a), \ldots, g_t(a))$. We have $\ker \phi = \bigcap_{j=1}^t \ker g_j$.  Hence, we have \begin{align*}
    g_1, \ldots, g_t \text{ generate } A^\vee &\iff \psi \text{ is surjective} \\
    &\iff \psi^\vee \text{ is injective} \\
    &\iff \phi = \psi^\vee \circ j \text{ is injective} \\
    &\iff \ker \phi = 0 \\
    &\iff \bigcap_j \ker g_j = 0 \\
    &\iff g_1, \ldots, g_t \text{ co-generate } A. \qedhere
\end{align*}
\end{proof}

\begin{rem}\label{rem:fcogArt} One could obtain as a corollary that a Matlis-dualizable module is finitely co-generated if and only if it is Artinian.  However, in \cite{Vam}  it is shown that over a Noetherian ring, this equivalence holds for \emph{any} module. Note that V\'amos uses the term ``finitely embedded'' for what we and others (see e.g. \cite{Lam99}) call ``finitely co-generated''.\end{rem}

\begin{lemma}
\label{lemma:cogensspan}
Let $(R,\m,k)$ be a Noetherian local ring.  
Let $A$ be an Artinian $R$-module, and let $g_1,\ldots,g_t \in A^\vee$. The following are equivalent:
\begin{enumerate}
    \item $g_1,\ldots,g_t$ is a co-generating set for $A$,
    \item The restrictions of the $g_i$ to $\soc A$ span $\Hom_R(\soc A,k)$ as a $k$-vector space.
\end{enumerate}
\end{lemma}

\begin{proof}
(1) $\Rightarrow$ (2): Let $h \in \Hom_R(\soc A,k)$, and let $j:k \to E$ be the natural inclusion. Since $E$ is injective, there exists a map $\tilde{h}:A \to E$ extending $j \circ h:\soc A \to E$. By Lemma \ref{lemma:cogeneration}, since the $g_1,\ldots,g_t$ co-generate $A$ and since $A$ is Artinian, we have that $g_1, \ldots, g_t$ generate $A^\vee$. Hence $\tilde{h}=\sum_{i=1}^t r_ig_i$ for some $r_i \in R$. Restricting to $\soc A$, we get
\[j \circ h = \tilde{h}|_{\soc A}=\sum_{i=1}^t (r_ig_i)|_{\soc A}=\sum_{i=1}^t r_i(g_i|_{\soc A}).\]
By Lemma \ref{lemma:imageinsocle}, the images of the $g_i|_{\soc A}$ are all inside of the copy of $k$ inside of $E$, and so $h=\sum_{i=1}^t \bar{r_i}(g_i|_{\soc A})$, as desired.

(2) $\Rightarrow$ (1): Set $B=\bigcap_i \ker(g_i) \subseteq A$. 
Let $x \in \soc B \subseteq \soc A$. Then $g_i|_{\soc A}(x)=0$ for all $1 \le i \le t$. Since the $g_i|_{\soc A}$ span $\Hom_R(\soc A,k)$ as a $k$-vector space, this implies that $x=0$. Hence $\soc B=0$, which by Remark~\ref{lem:simple} means that $B=0$, since $B$ is Artinian. Therefore, $g_1,\ldots,g_t$ is a co-generating set for $A$.
\end{proof}

\begin{prop}\label{lemma:cogentobasis}
Let $R$ be a Noetherian local ring, and let  $g_1,\ldots,g_t$ be a co-generating set of an Artinian $R$-module $A$. \tfae
\begin{enumerate}
    \item $g_1,\ldots,g_t$ is a minimal co-generating set for $A$,
    \item The restrictions of the $g_i$ to $\soc A$ form a basis for $\Hom_R(\soc A,k)$ as a $k$-vector space.
\end{enumerate}
\end{prop}

\begin{proof}
First we prove that (1) implies that the $g_i|_{\soc A}$ are linearly independent. We already know that they are a spanning set for $\Hom_R(\soc A,k)$ from Lemma \ref{lemma:cogensspan}. Suppose $\sum_{i=1}^t \bar{r_i}(g_i|_{\soc A})=0$, with the $\bar{r_i} \in k$ and at least one $\bar{r_j} \ne 0$. Without loss of generality, suppose that $\bar{r_t} \ne 0$, and by multiplying by an appropriate unit that $\bar{r_t}=-\bar{1}$. Then we can rewrite our equation as
\[g_t|_{\soc A}=\sum_{i=1}^{t-1} \bar{r_i}(g_i|_{\soc A}).\]

We have $\ker(g_t|_{\soc A})\subseteq \bigcap_{i=1}^{t-1} \ker(g_i|_{\soc A})$. Hence 
\[\bigcap_{i=1}^{t-1} \ker(g_i|_{\soc A})=\bigcap_{i=1}^t \ker(g_i|_{\soc A})=0,\]
which contradicts our hypothesis that $g_1,\ldots,g_t$ is a minimal co-generating set for $A$.

Next we prove that (2) implies that $g_1,\ldots,g_t$ are a minimal co-generating set for $A$. We already know that they are a co-generating set for $A$ by Lemma \ref{lemma:cogensspan}, so it suffices to prove minimality. Suppose without loss of generality that $\bigcap_{i=1}^{t-1} \ker(g_i)=0$. Then $g_1,\ldots,g_{t-1}$ also form a co-generating set for $A$, so they generate $A^\vee$. Hence $g_t=\sum_{i=1}^{t-1} r_ig_i$ for some $r_i \in R$. This implies that $g_t|_{\soc A}=\sum_{i=1}^{t-1} \bar{r_i}(g_i|_{\soc A})$, which contradicts the hypothesis that the $g_i|_{\soc A}$ are linearly independent.
\end{proof}

\begin{lemma}\label{lemma:oldlemma}
\cite[Theorem 2.3]{HRR-bf}
Let $(R,\m)$ be a Noetherian local ring and $L \subseteq M$ be \fg\ $R$-modules.  The following are equivalent.
\begin{enumerate}
    \item $L \cap \m M= \m L$.
    \item Any minimal generating set of $L$ extends to a minimal generating set for $M$.
\end{enumerate}
\end{lemma}

\begin{rem} 
We have \[\displaystyle\frac{L}{L \cap \m M}\cong \displaystyle\frac{L+ \m M}{\m M},\]  the second of which is the image of the vector space $L/\m L$ in $M/\m M$.
So (1) means the map $L/\m L \to M/\m M$ is injective.
\end{rem}

\begin{prop}\label{prop:maxexpmincogen}
Let $(R,\m)$ be a Noetherian local ring and $\int$ a Nakayama interior on Artinian $R$-modules.  Let $A \subseteq B$ Artinian $R$-modules.  Suppose that $C \subseteq D$ are $\int$-expansions of $A$ in $B$, with $D$ a maximal $\int$-expansion.  Then any minimal co-generating set of $B/D$ extends to a minimal co-generating set for $B/C$.
\end{prop}

\begin{proof}
Given the setup of the statement of the proposition, we have:
\[(B/D)^\vee \subseteq (B/C)^\vee \subseteq (B/A)^\vee \subseteq B^\vee,\]
with $(B/D)^\vee$ a minimal $\cl$-reduction of $(B/A)^\vee$ in $B^\vee$ by Theorem~\ref{thm:expansiondualisreduction}.
By Proposition~\ref{NR Ext}, any minimal set of generators of $(B/D)^\vee$ extends to a minimal set of generators for $(B/C)^\vee$, where said modules are being considered over $\hat{R}$. Given that a minimal set of generators for $(B/D)^\vee$ is a minimal cogenerating set for $B/D$ in $B$ by  Lemma~\ref{lemma:cogeneration}, and the same holds with $D$ replaced by $C$, this gives the desired result.
\end{proof}

\begin{lemma}
\label{lemma:dualofintersectionissum}
Let $R$ be a complete Noetherian local ring.  
Let $B$ be an $R$-module such that it and all of its quotient modules are Matlis-dualizable.  Let $\{C_i\}_{i \in I}$ a collection of submodules of $B$. Then
\[ \left(\frac{B}{\sum_i C_i}\right)^\vee \cong \bigcap_i (B/C_i)^\vee\] and 
\[ \left(\frac{B}{\bigcap_i C_i}\right)^\vee \cong \sum_i (B/C_i)^\vee,\]
where all the dualized modules are considered as submodules of $B^\vee$.
\end{lemma}

\begin{proof}
Note that elements of the left hand side of the first isomorphism are maps from $B \to E_R(k)$ whose kernel contains $\sum_i C_i$. Elements of the right hand side are maps from $B \to E_R(k)$ whose kernel contains $C_i$ for all $i \in I$, which proves the first isomorphism.  

For the second isomorphism, we will apply Matlis duality to the first isomorphism.  Namely, let $M= B^\vee$, and $L_i := (B/C_i)^\vee$ for each $i \in I$, considered as a submodule of $B^\vee$.  Then we have \begin{align*}
\frac B {\bigcap_i C_i} &= \frac{M^\vee} {\bigcap_i (M/L_i)^\vee} = \frac{M^\vee}{\left(M / \sum_i L_i\right)^\vee} \\
&= \left(\sum_i L_i\right)^\vee = \left(\sum_i (B/C_i)^\vee\right)^\vee.
\end{align*}
One more application of Matlis duality to the start and end of the chain of equalities then yields the second isomorphism.
\end{proof}

\begin{defn}
Let $R$ be an associative (not necessarily commutative) ring and $\int$ an interior operation defined on a class of (left) $R$-modules $\cM$.  If $A \subseteq B$ are elements of $\cM$, the \emph{$\int\hull$ of a submodule $A$ with respect to $B$} is the sum of all $\int$-expansions of $A$ in $B$, or
 \[
\int\hull^B(A):=\sum_{\int(C) \subseteq A \subseteq C \subseteq B} C.
\]
\end{defn}

\begin{thm}
\label{thm:hullexists}
Let $R$ be a complete Noetherian local ring.  Let $A \subseteq B$ be Artinian $R$-modules, and let $\int$ be a Nakayama interior defined on Artinian $R$-modules. Then the $\int$-hull of $A$ in $B$ is dual to the $\cl$-core of $(B/A)^\vee$ in $B^\vee$, where $\cl$ is the closure operation dual to $\int$.
\end{thm}

\begin{proof}
Let $M=B^\vee$ and $N=(B/A)^\vee$. We need to show that 
\[\left(M/\cl\core_M(N)\right)^\vee = \int\hull^B(A).\]
This follows from the definition of $\cl\core$, $\int\hull$, and Lemma \ref{lemma:dualofintersectionissum}.
\end{proof}

We conclude the section by exploring the dual to the concept  of the finitely generated version of a submodule selector (see Definition~\ref{def:fin}).  Moreover, we show how this dual notion relates to Theorem~\ref{thm:finintideal}.

\begin{defn}\label{def:cofin}
Let $R$ be a complete Noetherian local ring.  Let $\alpha$ be a submodule selector on a class of $R$-modules that is closed under taking submodules and quotient modules. For any fixed $R$-module $M \in \cM$, if $N$ is a submodule of $N$, for now we denote $\pi_N: M \onto M/N$ to be the natural surjection.  The \emph{Artinistic version} $\alpha^f$ of $\alpha$ is defined as \[
\alpha^f(M) := \bigcap \{\pi_N^{-1}(\alpha(M/N)) : M/N \text{ is 
Artinian}\}.
\]
\end{defn}

Recall (cf. Remark~\ref{rem:fcogArt}) that it is equivalent to take the intersection over all finitely co-generated quotients.

\begin{thm}\label{thm:cogen}
Let $R$ be a complete Noetherian local ring.  Let $\cM$ be a class of Matlis-dualizable $R$-modules that is closed under taking submodules and quotient modules.  Let $\alpha$ be a preradical on $\cM$.  Then $(\alpha_f)^\dual = (\alpha^\dual)^f$.
\end{thm}

\begin{proof}
Let $A \in \cM^\vee$.  We have \[
    (\alpha_f)^\dual(A) = \left(\frac{A^\vee}{\alpha_f(A^\vee)}\right)^\vee = \left(\frac{A^\vee}{\displaystyle \sum_{N\subseteq A^\vee,\ N \text{f.g.}}\alpha(N)}\right)^\vee.
\]
By Lemma~\ref{lemma:dualofintersectionissum} and Matlis duality, the above is equal to
\[
\bigcap_{N\subseteq A^\vee,\ N \text{f.g.}} (A^\vee / \alpha(N))^\vee = \bigcap_{N\subseteq A^\vee,\ N \text{f.g.}} \pi_{(A^\vee/N)^\vee}^{-1}((N / \alpha(N))^\vee).
\]
By the usual one-to-one correspondence between submodules of $A$ and quotient modules of $A^\vee$, this is equal to
\[
\bigcap_{B \subseteq A,\ (A/B)^\vee \text{f.g.}} \pi_B^{-1}\left( \left(\frac{(A/B)^\vee}{\alpha((A/B)^\vee)}\right)^\vee \right) = \bigcap_{B \subseteq A,\ (A/B)^\vee \text{f.g.}} \pi_B^{-1}(\alpha^\dual(A/B)).
\]
But by Definition~\ref{def:cofin} and Lemma~\ref{lemma:cogeneration}, the latter equals $(\alpha^\dual)^f(A)$.
\end{proof}

Hence, Theorem~\ref{thm:finintideal} can be reinterpreted as a statement about the Artinistic version of the interior operation dual to a closure operation. In the particular case of tight closure, it allows us to extend the 
interpretation of the big test ideal in terms of maps from $R^{1/p^e}$ 
developed in \cite{nmeSc-tint} to a comparable interpretation of the finitistic tight closure test ideal, as follows:

\begin{thm}
\label{thm:testidealfinitistic}
Let $R$ be a complete Noetherian local $F$-finite reduced ring of prime characteristic $p>0$.  Let $c$ be a big test element for $R$.  Then the finitistic test ideal of $R$ consists of those elements $a\in R$ such that for every $\m$-primary ideal $J$ of $R$ and every nonnegative integer $e\geq 0$, there is an $R$-linear map $g: R^{1/{p^e}} \ra R/J$ with $g(c^{1/p^e}) = a+J$.
\end{thm}

\begin{proof}
Let $\alpha(M)=0_M^*$ for any $R$-module $M$.
Let $c\in R$ be a big test element for $R$.  We have the following sequence of equalities, which is justified below:
\begin{align*}
    \tau_{fg}(R) &= (\alpha_f)^\dual(R) = (\alpha^\dual)^f(R)\\
    &=\bigcap \{\pi_I^{-1}(\alpha^\dual(R/I)) : R/I \text{ finitely co-generated}\}\\
    &= \bigcap \{\pi_I^{-1}(\alpha^\dual(R/I)) : \lambda(R/I)<\infty\} \\
    &= \bigcap \{\pi_I^{-1}((R/I)_{*R}) : \lambda(R/I)<\infty\} \\
    &= \{a\in R : a+I \in (R/I)_{*R} \text{ whenever } \lambda(R/I)<\infty\} \\
    &= \{a\in R : a+I \in \tr_{c^{1/q}, R^{1/q}}(R/I),\ \forall \text{ finite colength } I,\ \forall q \}.
\end{align*}

We justify the steps of this proof one by one.  The first equality is by Theorem~\ref{thm:finintideal} (or \cite[Theorem 5.5]{nmeRG-cidual}).  The second equality is by Theorem~\ref{thm:cogen}.  The equality on the second line is by definition.

To see the equality on the third line, note that $R/I$ is finitely co-generated if and only if $(R/I)^\vee$ is finitely generated (by Lemma~\ref{lemma:cogeneration}).  But $(R/I)^\vee \cong \ann_E(I)$ is finitely generated if and only if it is of finite length, since it is already Artinian.  But of course $(R/I)^\vee$ has finite length if and only if $\len(R/I)<\infty$.

The equality on the fourth line follows from \cite[Corollary 3.6]{nmeSc-tint}, which in our terminology says that tight interior (in the sense given in \cite{nmeSc-tint}) is smile-dual to tight closure.  The equality on the fifth line is by definition of $\pi_I$.  The equality on the final line then follows from \cite[Theorem 2.5]{nmeSc-tint}.
\end{proof}

\section{Core and hull comparisons for known closure operations and their dual interiors} \label{sec:chc}

In this section we extend many of the results of \cite{FoVa-core} on $\cl$-spread to the module setting, proving that liftable integral spread exists along the way, and then prove dual results for $\int$-hulls.

\begin{defn}
Let $R$ be an associative, not necessarily commutative ring. If $\cl_1$ and $\cl_2$ are both closure operations on a class of left $R$-modules $\cM$, we say that $\cl_1 \leq \cl_2$ if $N^{\cl_1}_M \subseteq N^{\cl_2}_M$ for all $R$-modules $M \in \cM$ and all submodules $N \subseteq M$.  
\end{defn}

The following proposition generalizes \cite[Lemma 3.3]{FoVa-core} to the module setting.

\begin{prop}\label{prop:clcontainment}
Let $R$ be a local ring. If $\cl_1 \leq \cl_2$ are closure operations defined on the class of \fg\ $R$-modules $\cM$ with $\cl_2$ Nakayama, and if $L$ is a $\cl_1$-reduction of $N$ in $M$, then there exists a minimal $\cl_2$-reduction $K$ of $N$ with $K \subseteq L$.
\end{prop}

\begin{proof}
Notice for all $\cl_1$-reductions $L$ of $N$ in $M$, $L\subseteq N \subseteq  L^{\cl_1}_M$.  
Since $\cl_1 \leq \cl_2$, $L^{\cl_1}_M \subseteq L^{\cl_2}_M$ for all submodules $L \subseteq M$.  
Hence $L \subseteq N \subseteq L^{\cl_1}_M\subseteq L^{\cl_2}_M.$  So $L$ is a $\cl_2$-reduction of $N$.  Now by Proposition \ref{NR Ext}, there is a minimal $\cl_2$-reduction $K \subseteq L$ of $N$ in $M.$
\end{proof}

As \cite[Proposition 3.4]{FoVa-core} does for ideals, we use this result 
to establish a containment between the $\cl_2$-core and the $\cl_1$-core of $N$ in $M$.

\begin{prop}\label{prop:clcorecontainment}
Let $R$ be a local ring and $\cl_1 \leq \cl_2$ be closure operations defined on the class of \fg\ $R$-modules $\cM$ with $\cl_2$ Nakayama.  If $N \subseteq M$ are $R$-modules in $\cM$, then ${\cl_2}\core_M(N) \subseteq {\cl_1}\core_M(N)$.
\end{prop}

\begin{proof}
For any submodule $N \subseteq M$ in $\cM$, $\cl_1\core_M(N)=\bigcap_{L_1 \subseteq N \subseteq (L_1)_M^{\cl_1}} L_1$.  
By Proposition \ref{prop:clcontainment}, for every $\cl_1$-reduction of $N$ in $M$, there is a minimal $\cl_2$-reduction $L_2$ of $N$ in $M$ such that $L_2 \subseteq L_1.$  
Now 
\[{\cl_2}\core_M (N) \subseteq \bigcap_{L_2 \subseteq L_1 \subseteq N \subseteq (L_2)_M^{\cl_2}} L_2,\]
and 
\[\bigcap_{L_2 \subseteq L_1 \subseteq N \subseteq (L_2)_M^{\cl_2}} L_2 \subseteq \bigcap_{L_1 \subseteq N \subseteq (L_1)_M^{\cl_1}} L_1={\cl_1}\core_M(N). \qedhere\]
\end{proof}

The next corollary extends Corollary 3.5 of \cite{FoVa-core} to the module setting.

\begin{cor} Let $R$ be a Noetherian local ring of characteristic $p$, and let $N \subseteq M$ be $R$-modules.  Then
\[\leftfootline\core_M(N) \subseteq\ *\core_M (N) \subseteq\ F\core_M (N).
\]
\end{cor}

\begin{proof}
It is clear from the framework where Frobenius closure is first introduced \cite[Section 10]{HHmain} that the Frobenius closure of a submodule is always contained in its tight closure. That the tight closure of a submodule is always contained in its liftable integral closure follows from \cite[Proposition 2.4 (5)]{nmeUlr-lint}.
Now the result follows directly from Proposition \ref{prop:clcorecontainment} since \[N^F_M \subseteq N^*_M \subseteq N^{\leftfootline}_M
\] 
 for all $R$-submodules $N \subseteq M$.
\end{proof}

\begin{defn}
Let $R$ be a Noetherian local ring and $M$ a \fg\ $R$-module such that $\cl$ is defined on submodules of $M$. A submodule $N \subseteq M$ is said to have $\cl$-spread 
if all the minimal $\cl$-reductions of $N$ in $M$ have the same minimal number of generators. In this case, we denote this common number by $\ell^{\cl}_M(N)$ and call it the $\cl$-spread of $N$ in $M$.
\end{defn}  

Next we extend \cite[Proposition 3.7]{FoVa-core} to the module setting.

\begin{prop}\label{spreadcomp}
Let $(R, \m)$ be a Noetherian local ring.  Let $\cl_1 \leq \cl_2$ be Nakayama closure operations defined on a class of \fg\ $R$-modules $\mathcal{M}$.  If $N \subseteq M \in \mathcal{M}$ and the $\cl_1$- and $\cl_2$-spread of $N$ in $M$ exist, then $\ell^{\cl_1}_M(N) \geq \ell^{\cl_2}_M(N)$.
\end{prop}

\begin{proof}
Let $L$ be a minimal $cl_1$-reduction of $N$ in $M$.  Then $\mu(L) = \ell_M^{\cl_1}(N)$.  Since $L \subseteq N \subseteq {L}^{\cl_1}_M \subseteq {L}^{\cl_2}_M$, then $L$ is a $\cl_2$-reduction of $N$ in $M$ (but not necessarily a minimal $\cl_2$-reduction of $N$).
By Proposition \ref{NR Ext}, $\ell_M^{\cl_1}=\mu(L) \geq \ell^{\cl_2}_M(N)$.
\end{proof}

\begin{rem}
\label{rem:otherspreadsexist}
This leads to the question: when does the $\cl$-spread exist? 
The first named author showed in \cite{nme-sp} that whenever $R$ is an excellent and analytically irreducible local domain of prime characteristic $p>0$ then for all $N \subseteq M$ with $M$ \fg, both the $*$-spread $\ell^*_M(N)$ and the $F$-spread $\ell^F_M(N)$ exist. In the next result we prove that the liftable integral spread typically exists and agrees with the analytic spread, the integral closure spread originally defined for ideals.
\end{rem}

\begin{thm}\label{thm:lispread}
Let $(R,\m,k)$ be a Noetherian local ring such that $k$ is infinite. Assume that either $R$ is $\mathbb Z$-torsion free (e.g. if it is of equal characteristic 0) or that it is unmixed and generically Gorenstein (e.g. if it is reduced).  Then the liftable integral spread exists.

In particular, if $L \subseteq M$ are finitely generated $R$-modules and $\pi: F \onto M$ is a minimal surjection from a finitely generated free module $F$, then the liftable integral spread of $L$ in $M$ is the analytic spread of $\pi^{-1}(L)$ in the sense of Eisenbud-Huneke-Ulrich.
\end{thm}

\begin{proof}
First assume that $M$ is itself a \fg\ free module, so that $\pi$ is the identity map.  By \cite[Theorem 0.3]{EHU-Ralg},  integrality in their sense coincides with integrality within the symmetric algebra of $M$ under our hypotheses. That is, say $M = \sum_{i=1}^n R x_i$, where the $x_i\in M$ are linearly independent over $R$. Then any submodule $U$ of $M$ is generated by $R$-linear combinations of the $x_i$.  Denote by $R[U]$ the $R$-subalgebra of the polynomial ring $R[x_1, \ldots, x_n]$ generated by a generating set for $U$.  Then for submodules $U \subseteq L \subseteq M$, we have $L \subseteq U^\li_M$ if and only if the induced map $R[U] \ra R[L]$ is module-finite.

Let $\ell=\dim (R[L] / \m R[L])$.  We claim that every minimal $\li$-reduction of $L$ in $M$ is $\ell$-generated.  Since we know minimal $\li$-reductions exist (see Proposition~\ref{NR Ext}), it will be enough to show the following two things: \begin{enumerate}
    \item\label{it:atleastell} No $\li$-reduction of $L$ in $M$ can be generated by fewer than $\ell$ elements, and
    \item\label{it:containsell} Every $\li$-reduction of $L$ in $M$ contains a $\li$-reduction that is $\ell$-generated.
\end{enumerate}
First suppose that $L$ contains a $\li$-reduction $U$ of $M$ with $\mu(U) = t<\ell$.  Then $R[L]$ is module-finite over $R[U]$ (which is $t$-generated as an $R$-algebra), whence $R[L] / \m R[L]$ is module-finite over $R[U] / (\m R[L] \cap R[U])$.  But the latter is at most $t$-generated as a $k$-algebra, so $\dim (R[U] / (\m R[L] \cap R[U])) \leq t$.  On the other hand, the module-finiteness shows that $\dim (R[U] / (\m R[L] \cap R[U])) = \ell >t$, and we have a contradiction that proves (\ref{it:atleastell}).

For (\ref{it:containsell}), let $U$ be a $\li$-reduction of $L$ in $M$.  We have that $A := R[U] / (\m R[L]\cap R[U])$ is a standard graded $k$-algebra, with $k$ an infinite field, and by the above reasoning about module-finiteness, we have $\dim A = \ell$.  Then by the graded Noether normalization theorem \cite[Theorem 1.5.17]{BH}, there exist algebraically independent degree one elements $a_1, \ldots, a_\ell \in A$ such that $A$ is module-finite over $k[a_1, \ldots, a_\ell]$.  Choose elements $u_j\in U$ whose residue classes mod $\m R[L]\cap R[U]$ are the $a_j$.  Let $V := \sum_{j=1}^\ell R u_j$.  By the module-finiteness condition, we have $A_s = k[a_1, \ldots, a_\ell]_s$ for $s\gg 0$, thought of as finite dimensional vector spaces over $k$.  By Nakayama’s lemma, it follows that $R[V]_s = R[U]_s$ for $s\gg 0$, whence $R[U]$ is module-finite over $R[V]$.  Since module-finiteness is transitive, $R[L]$ is also module-finite over $R[V]$.  Hence $L$ has an $\ell$-generated $\li$-reduction contained in $U$, namely $V$.

Now we examine the general case.  Let $\pi: F \onto M$ be as in the statement of the theorem, so that $\ker \pi \subseteq \m F$.  Let $\tilde{L} = \pi^{-1}(L)$, and let $U$ be a minimal $\li$-reduction of $L$ in $M$.  Then $\tilde{U}=\pi^{-1}(U)$ is a minimal $\li$-reduction of $\tilde{L}$ in $F$.  Hence by the above, we have $\tilde{U} \cap \m F = \m \tilde{U}$, so since $\ker \pi \subseteq \m F$, we have $\tilde{U} \cap \ker \pi \subseteq \m \tilde{U}$.  Thus, since $U \cong \tilde{U} / (\ker \pi \cap \tilde{U})$, we have \begin{align*}
\mu(U) &= \dim_k(U/\m U) = \dim_k\left(\frac{\tilde{U} / (\ker \pi \cap \tilde{U})}{\m(\tilde{U} / (\ker \pi \cap \tilde{U}))}\right)\\
&=\dim_k(\tilde{U} / (\m \tilde{U} + (\ker \pi \cap \tilde{U}))) = \dim_k(\tilde{U} / \m \tilde{U}) = \mu(\tilde{U}). \qedhere
\end{align*}
\end{proof}

Now we can extend Corollary 3.8 of \cite{FoVa-core} to modules, using liftable integral closure as our integral closure on modules.

\begin{cor}
Let $(R, \m)$ be an excellent, analytically irreducible domain of characteristic $p>0$, with infinite residue field.  Then for all finitely generated modules $N \subseteq M$, 
\[\ell^{\leftfootline}_M(N) \leq \ell^*_M (N) \leq \ell^F_M(N).\]
\end{cor}

Now we explore similar results on expansions and hulls, using the duality built up in Section \ref{sec:expansions}.
In particular, we discuss relationships between expansions and hulls for  the interior operations dual to Frobenius, tight, and liftable integral closure. 

\begin{defn}
Let $R$ be an associative but not necessarily commutative ring. Let $\int_1$ and $\int_2$ be interior operations defined on a class $\cM$ of $R$-modules. We say that $\int_1 \leq \int_2$ if for all $M \in \cM$, ${\int_1}(M) \subseteq {\int_2}(M)$.
\end{defn}

Using the notion of Nakayama interior, we derive similar statements to Propositions \ref{prop:clcontainment} and \ref{prop:clcorecontainment} for the containments of $\int$-expansions and $\int$-hulls.

\begin{prop}\label{prop:expcontainment}
Let $(R,\m)$ be a Noetherian local ring and $\int_2 \leq \int_1$ be interior operations on the class of Artinian $R$-modules $\cM$ with $\int_2$ a Nakayama interior then any $\int_1$-expansion of $A$ in $B$ is contained in a maximal $\int_2$-expansion of $A$ in $B$.
\end{prop}

\begin{proof}
Suppose $A \subseteq C \subseteq B$ and $C$ is an $\int_1$-expansion of $A$ in $B$.  Thus $\int_1(C) \subseteq A \subseteq C$.  Since $\int_2(C) \subseteq \int_1(C) \subseteq A \subseteq C$, $C$ is also an $\int_2$-expansion of $A$ in $B$.  Finally, $C$ must be contained in a maximal $\int_2$-expansion of $A.$  
\end{proof}

\begin{prop}\label{prop:hullcontainment}
Let $R$ be an associative (i.e. not necessarily commutative) ring and $\int_2 \leq \int_1$ interior operations on a class $\cM$ of (left) $R$-modules.  Let $A \subseteq B$ be $R$-modules such that $\int_1$ and $\int_2$ are defined on all $R$-modules between $A$ and $B$.  Then  $\int_1\hull^B(A) \subseteq \int_2\hull^B(A)$.
\end{prop}

\begin{proof}
Let $C$ be an $\int_1$-expansion of $A$ in $B$.  Then we have 
\[\int_2(C) \subseteq \int_1(C) \subseteq A \subseteq C,\] whence $C$ is an $\int_1$-expansion of $A$ in $B$.  Hence \begin{align*}
\int_1\hull^B(A) &= \sum_{i_1(C) \subseteq A \subseteq C \subseteq B} C \\
&\subseteq \sum_{i_2(D) \subseteq A \subseteq D \subseteq B}
D = \int_2\hull^B(A). \qedhere
\end{align*}
\end{proof}

\begin{cor}
Let $(R, \m)$ be a Noetherian local $F$-finite ring of characteristic $p$ and $A \subseteq B$ Artinian $R$-modules, then 
\[F\hull^B (A) \subseteq *\hull^B (A) \subseteq \leftfootline\hull^B (A).\]
\end{cor}

\begin{proof}
First note that the first named author and Ulrich showed that $* \leq \leftfootline$ \cite{nmeUlr-lint} as closure operations.  
Now by \cite[Proposition 7.5]{nmeRG-cidual}, $\leftfootline^{\smile} \leq *^{\smile} \leq F^{\smile}$ or in other words, the liftable integral interior is less than or equal to the star interior, which is less than or equal to the Frobenius interior.  
Hence by Proposition \ref{prop:hullcontainment} we obtain $F\hull^B (A) \subseteq *\hull^B (A) \subseteq\ \leftfootline\hull^B (A)$.
\end{proof}

We discuss cases where we can say more about the integral and $*$-hull.

\begin{prop}
Let $(R,\m)$ be a complete Noetherian local equidimensional ring having no embedded primes.  Assume $\dim R \geq 2$. Then any finitely generated free module $F$ has liftable integral interior equal to zero.  Hence, for any finitely generated free module $F$ and any submodule $L \subseteq F$, the liftable integral hull of $L$ in $F$ is $F$.  In particular, the liftable integral hull of any ideal is the unit ideal.
\end{prop}

\begin{proof}
Let $\alpha(M)$ denote the liftable integral closure of $0$ in $M$, and $\beta$ be its dual interior operation. Let $F=R^t$ be a free module of rank $t$. By \cite[Theorem 5.1]{nmeUlr-lint}, we have $0^\li_E = E$. Then
\[
\beta(F) = \alpha^\dual(F) = \left(\frac{F^\vee}{\alpha(F^\vee)}\right)^\vee = \left(\frac{E^t}{\alpha(E)^t}\right)^\vee = 0^\vee = 0. \qedhere
\]
\end{proof}

Next, we note the following general fact about the interior of a local ring.

\begin{lemma}\label{lem:lrint}
Let $(R,\m)$ be a local (not necessarily Noetherian) commutative ring and let $\int$ be an interior operation defined at least on ideals of $R$.  Let $J$ be an ideal of $R$.  Then $\int(R) \subseteq J \iff \int\hull(J) = R$.
\end{lemma}

\begin{proof}
First, if $\int(R)\subseteq J$, then $R$ is an $\int$-expansion of $J$, whence $\int\hull(J) \supseteq R$.

Conversely, suppose $\int\hull(J)=R$.  Then $1\in \int\hull(J)$, the sum of the $\int$-expansions of $J$, whence there are $\int$-expansions $K_1, \ldots, K_t$ of $J$ such that $1 \in \sum_{j=1}^t K_j$.  Say $1=\sum_{j=1}^t a_j$, $a_j \in K_j$.  Then there is some $i$ with $a_i \notin \m$ (otherwise the sum is in $\m$, but the sum is $1$), whence $a_i$ is a unit and $K_i = R$.  That is, $R$ is an $\int$-expansion of $J$, which means that $\int(R) \subseteq J$.
\end{proof}

In dimension 1, we have the following consequence:

\begin{prop}
Let $(R,\m)$ be a 1-dimensional Cohen-Macaulay approximately Gorenstein complete Noetherian local ring, with infinite residue field.  Let $J$ be an ideal of $R$.  Let $C_R$ be the conductor of $R$.  Then $C_R \subseteq J \iff \li\hull(J) = R$.
\end{prop}

\begin{proof}
By \cite[Theorem 4.4]{nmeUlr-lint}, $C_R$ is the uniform annihilator of the modules $I^-/I$ for all ideals $I$ of $R$.  Since any integrally closed ideal is the intersection of integrally closed $\m$-primary ideals \cite[Corollary 6.8.5]{HuSw-book}, $C_R$ is then also equal to the uniform annihilator of the modules $I^-/I$ for finite colength ideals $I$ of $R$ by \cite[Theorem 5.5]{nmeRG-cidual}.  Since $R$ is approximately Gorenstein, an appeal to Theorem~\ref{thm:finintideal}, along with the fact that liftable integral closure is finitistic \cite[Lemma 2.3]{nmeUlr-lint}, shows that $C_R$ is the liftable integral interior of $R$.  
An application of Lemma~\ref{lem:lrint} finishes the proof.
\end{proof}

In the next result, $\tau$ will denote the tight closure test ideal of $R$.

\begin{prop}\label{prop:starhull}
Let $(R, \m )$ be an $F$-finite ring of characteristic $p>0$, then $*\hull (I)=R$ if and only if $\tau \subseteq I$.  
\end{prop}

\begin{proof}
By \cite[Proposition 2.3]{nmeSc-tint}, the tight interior of $R$ is $\tau$.  Then the result follows from Lemma`\ref{lem:lrint}..
\end{proof}

\begin{defn}
Let $(R, \m)$ be a Noetherian local ring.  Let $\int$ be an interior operation defined on a class of Artinian $R$-modules $\mathcal{M}$.  Let $A \subseteq B$ be Artinian $R$-modules.  We define the \emph{$\int$-co-spread} $\ell_{\int}^B(A)$ of $A$ to be the minimal number of cogenerators of $B/C$ of any maximal $\int$-expansion $C$ of $A$, if this number exists.
\end{defn}

\begin{prop}
Let $(R,\m)$ be a Noetherian local ring and $\int$ a Nakayama interior operation defined on a class of Artinian $R$-modules. Let $\cl$ be the closure operation dual to $i$. Let $A \subseteq B$ be Artinian $R$-modules, $M=B^\vee$, and $N=(B/A)^\vee$. If the $\cl$-spread $\ell_M^{\cl}(N)$ of $N$ in $M$ exists, then the $\int$-co-spread $\ell_i^B(A)$ of $A$ in $B$ exists.
In particular, the tight interior co-spread and Frobenius interior co-spread exist under the hypotheses of Remark \ref{rem:otherspreadsexist} and the liftable integral interior co-spread exists under the hypotheses of Theorem \ref{thm:lispread}.
\end{prop}

\begin{proof}
Let $C$ be a maximal $\int$-expansion of $A$ in $B$. By the proof of Proposition \ref{prop:maxintexpansions}, $(B/C)^\vee$ is a minimal $\cl$-reduction of $N$ in $M$. Since the $\cl$-spread of $N$ in $M$ exists, $(B/C)^\vee$ is minimally generated by $\ell_M^{\cl}(N)$ elements. By Lemma \ref{lemma:cogeneration}, $B/C$ is minimally co-generated by $\ell_M^{\cl}(N)$ elements. Since this holds for every maximal $\int$-expansion $C$ of $A$ in $B$, the $\int$-co-spread $\ell_i^B(A)$ exists.

The last sentence of the Proposition follows immediately.
\end{proof}

\begin{prop}\label{prop:shrinkcomp}
Let $(R, \m)$ be a Noetherian local ring and $\int_2 \leq \int_1$ be Nakayama interior operations on the class of Artinian $R$-modules $\mathcal{M}$ such that $\ell_{i_1}^B(A)$ and $\ell_{i_2}^B(A)$ exist. Then $\ell_{\int_2}^B(A) \leq \ell_{\int_1}^B(A)$.
\end{prop}

\begin{proof}
Suppose $A \subseteq C \subseteq B$ and $C$ is a maximal $\int_1$-expansion of $A$ in $B$.  Then $\ell_{\int_1}^B(A)$ is the minimal number of cogenerators of $B/C$.  By Proposition \ref{prop:expcontainment}  there exists $A \subseteq C \subseteq D \subseteq B$ with $D$ a maximal $\int_2$-expansion of $A$. Thus any minimal cogenerating set of $B/D$ extends to a minimal cogenerating set of $B/C$ by Proposition \ref{prop:maxexpmincogen}.  Hence $\ell_{\int_1}^B(A) \geq \ell_{\int_2}^B(A)$.
\end{proof}

\begin{cor}
Let $(R, \m)$ be a Noetherian local ring of \charp\ satisfying the hypotheses of both Remark \ref{rem:otherspreadsexist} and Theorem \ref{thm:lispread} and
$A \subseteq B$ Artinian $R$-modules, then $\ell_F^B (A) \leq \ell_*^B (A) \leq \ell_{\leftfootline}^B(A)$.
\end{cor}

\section{Computations of interiors and hulls}
\label{sec:examples}

To illustrate Theorem \ref{thm:finintideal} in action, we construct the tight and Frobenius interiors of some ideals in certain nice rings of prime characteristic.  Having done this, we then compute the hulls of some of these ideals.

In order to use the theorem to compute tight interiors, we need a result telling us when the tight interior equals its Artinistic version.  To that end, we repurpose work from two papers of Lyubeznik and Smith from around the turn of the century.

\begin{thm}\label{thm:LScases}
Let $(R,\m)$ be a complete reduced $F$-finite local ring, and $I$ an ideal of $R$.  Suppose either \begin{enumerate}
    \item\label{it:graded} There is a positively graded $\N$-graded algebra $A$ over a field $K$, with graded maximal ideal $\n$, and a homogeneous ideal $J$ of $A$, such that $R=\widehat{A_\n}$ and $I=JR$, or
    \item\label{it:isosing} $R$ is an isolated singularity.
\end{enumerate}
Then $\ann_E(I)^*_E = \ann_E(I)^{*fg}_E$.  Hence the tight interior of $I$ and the Artinistic version of the tight interior of $I$ coincide.
\end{thm}

\begin{proof}
First we show that in both cases (\ref{it:graded}) and (\ref{it:isosing}), we have $(\ann_EI)^*_E = (\ann_EI)^{*fg}_E$.

First assume we are in case (\ref{it:graded}).  Clearly $K$ must be $F$-finite.  By \cite[Theorem 3.3]{LySmFreg}, we have $(\ann_{E_A(A/\n)}J)^*_{E_A(A/\n)} = (\ann_{E_A(A/\n)}J)^{*fg}_{E_A(A/\n)}$ (using here also the fact that $\ann_{E_A(A/\n)}J$ must be a graded submodule of $E_A(A/\n)$).  But $E = E_A(A/\n)$, and by Hom-tensor adjointness we have $\ann_EI = \ann_EJ$. Set $D$ to be this module. The tight closure and finitistic tight closure of $0$ in $G$ as $A$-modules coincide. By persistence of tight closure this tight closure is contained in the finitistic tight closure of $0$ in $G$ as an $R$-module. In turn, this is contained in the ordinary tight closure of $0$ in $G$ as an $R$-module. Hence it is enough to show that the latter is equal (as a subset of $G$) 
to the tight closure of $0$ in $G$ as an $A$-module.  But since $G := E/D$ is Artinian, every element of $A \setminus \n$ acts as a unit on it, and its $A_\n$-module structure coincides with its $R$-module structure.  Hence, we have our result.

On the other hand, suppose we are in case (\ref{it:isosing}).  Let $D := \ann_EI$.  By \cite[Theorem 8.12]{LySm-Test}, $0^*_{E/D} = 0^{*fg}_{E/D}$.  Since both $*$ and $*fg$ are residual closures, it follows that $D^*_E = D^{*fg}_E$, as desired.

Now we can prove the last sentence of the Theorem. In either case, setting $\alpha(-) := 0^*_{-}$, we have \begin{align*}
    I_* &= \alpha^\dual(I) = \ann_R((\ann_EI)^*_E) = \ann_R((\ann_EI)^{*fg}_E) \\
    &= (\alpha_f)^\dual(I) = (\alpha^\dual)^f(I),
\end{align*}
where the first equality comes from \cite[Corollary 3.6]{nmeSc-tint}.  The second and fourth equalities follow from Theorem~\ref{thm:finintideal}.  The third equality follows from the earlier parts of the proof.  Finally, the last equality follows from Theorem~\ref{thm:cogen}.
\end{proof}

Our first few examples are in numerical semigroup rings, which are approximately Gorenstein.  Following \cite[Page 177]{HoFNDTC}, in the following examples we construct irreducible ideals $J_n$ cofinal with the maximal ideal in order to compute the double colons needed to determine the finitistic interior related to a residual closure operation.

\begin{example}
Let $k$ be an infinite $F$-finite field of characteristic $p>3$ and let $R=k[[t^2,t^3]]$  with maximal ideal $\m=(t^2,t^3)$.  We use Theorem \ref{thm:finintideal} to compute the $*$-interior of all ideals of $R$ and then compute the $*$-hulls of all ideals of $R$. 

 The nonzero, non-unital ideals of $R$ are either of the form $(t^m,t^{m+1})$ or $(t^m+at^{m+1})$ where $m \ge 2$ and $a \in k$. The lattice of ideals includes the following:

$$
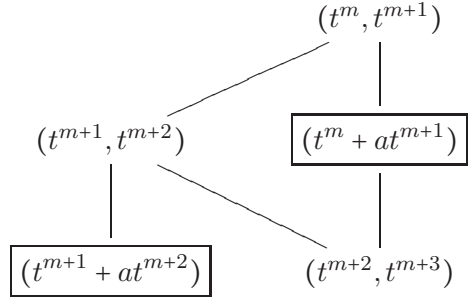
\xymatrix{ & (t^m,t^{m+1}) \ar @{-} [dl] \ar @{-} [d]  \\
            (t^{m+1},t^{m+2})\ar @{-} [dr]\ar @{-} [d] &\boxed{(t^m+at^{m+1})}\ar @{-} [d]  \\
             \boxed{(t^{m+1}+at^{m+2})} & (t^{m+2},t^{m+3}) \\}$$
where the boxed ideals contain $|k|$ ideals.

Let $J_n=(t^{2n}+at^{2n+1})$ for $n \ge 1$. Then the $J_n$ are a decreasing sequence of irreducible ideals cofinal with the powers of the maximal ideal. They are irreducible because $R$ is Gorenstein and $t^{2n}+at^{2n+1}$ is a regular element, whence $R/(t^{2n}+at^{2n+1})$ is Gorenstein.  By Theorem \ref{thm:finintideal}, this implies that for any ideal $I$ of $R$, the Artinistic tight interior of $I$ is equal to 
\[\bigcap_{n \ge 1} J_n:(J_n:I)^*_R.\]

Since $R$ is an isolated singularity, Theorem~\ref{thm:LScases} guarantees that Artinistic and ordinary tight interior agree on ideals of $R$.

For the following computations, we note that 
\begin{align*}(t^{n}+at^{n+1}):(t^m,t^{m+1})&=(t^{n}+at^{n+1}):(t^m) \cap (t^{n}+at^{n+1}):(t^{m+1})\\
&=(t^{n-m}+at^{n-m+1}) \cap (t^{n-m-1}+at^{n-m})\\
&=(t^{n-m+2},t^{n-m+3}). \\
\end{align*}

Since $R$ is a 1-dimensional domain with infinite residue field, it is known \cite[Example 1.6.2]{HuTC} that  $(t^{m})^*=(t^m)^-=(t^m,t^{m+1})$ for any $m\geq 2$.  For any $a \in k$, we compute $(t^m+at^{m+1})_*$ using Theorems~\ref{thm:finintideal} and \ref{thm:LScases}. 

\begin{align*}
(t^m+at^{m+1})_* &=\bigcap_{n \ge 1} J_n:(J_n:(t^m+at^{m+1}))^*\\
&= \bigcap_{n \ge 1} (t^{2n}+at^{2n+1}):((t^{2n}+at^{2n+1}):(t^m+at^{m+1}))^* \\.\end{align*}
For $2 \leq m<2n-1$, $(t^{2n}+at^{2n+1}):(t^m+at^{m+1})=(t^{2n-m})$, and given that the intersection is over ideals that decrease as $n$ increases  the above is equal to
\begin{align*}
\bigcap_{n \ge 1} (t^{2n}+at^{2n+1}):(t^{2n-m})^*
&=\bigcap_{n \ge 1} (t^{2n}+at^{2n+1}):(t^{2n-m},t^{2n-m+1}) \\
&=\bigcap_{n \ge m} (t^{2n}+at^{2n+1}):(t^{2n-m},t^{2n-m+1})\\
&=\bigcap_{n \ge m} (t^{m+2},t^{m+3})=(t^{m+2},t^{m+3}),
\end{align*}
which agrees with the the computations in \cite{Va-inthull}.

We similarly compute $(t^m,t^{m+1})_*$ with $J_n=(t^{2n})$.
\begin{align*}
(t^m,t^{m+1})_* &=\bigcap_{n \ge 1} J_n:(J_n:(t^m,t^{m+1}))^*\\
&= \bigcap_{n \ge 1} (t^{2n}):((t^{2n}):(t^m,t^{m+1}))^* \\
&= \bigcap_{n \ge 1} (t^{2n}):(t^{2n-m+2},t^{2n-m+3}) \\
&= \bigcap_{n \ge 1} (t^{m},t^{m+1}) = (t^m, t^{m+1}).\end{align*}

Next we compute the $*\hull$ of the ideals $(t^{m+2},t^{m+3})$ for $m \ge 2$.  Observing the lattice of ideals and using the computations above, $(t^m+at^{m+1})_*=(t^{m+2},t^{m+3})$ for all $a \in k$. 
Hence $(t^m+at^{m+1})$ is a maximal $*$-expansion of $(t^{m+2},t^{m+3})$ since $(t^m,t^{m+1})_*=(t^m,t^{m+1})$.  This implies that
\[*\hull(t^{m+2},t^{m+3})=\sum\limits_{a \in k} (t^m+at^{m+1})=(t^{m},t^{m+1})\] for $m \geq 2$.  As the test ideal of $R$ is $(t^2,t^3)$, Proposition \ref{prop:starhull} implies that the $*\hull(t^2,t^3)=R$.
Note that $*\hull(t^3,t^4)=(t^3,t^4)$ and  $*\hull(t^m+at^{m+1})=(t^m+at^{m+1})$ for $m \geq 2$ since any ideal lying directly above these ideals in the lattice has a different $*$-interior.

\end{example}

\begin{example}
Let $k$ be an $F$-finite field of characteristic $p>5$ and $R=k[[t^3,t^4,t^5]]$ with maximal ideal  $\m=(t^3,t^4,t^5)$.  Since $R$ is an isolated singularity, Theorem~\ref{thm:LScases} guarantees that Artinistic and ordinary tight interior of ideals agree in $R$. We use Theorem \ref{thm:finintideal} to compute the $*$-interior of principal ideals of $R$ and then compute the $*$-hulls of certain ideals of $R$. 

Unlike in the previous example, $R$ is not Gorenstein, but $\omega\cong (t^4,t^5)$ and taking $w=t^4$ in Hochster's construction \cite[Page 177]{HoFNDTC}, we have
\begin{align*}R/J_n &\cong (t^4)/((t^{3n}+at^{3n+1}+bt^{3n+2})(t^4,t^5) \cap (t^4) \\
&\cong (t^4)/(t^{3n+4}+at^{3n+5}+bt^{3n+5},t^{3n+5}+at^{3n+6}+bt^{3n+6})\\
\end{align*}
Thus $J_n=(t^{3n}+at^{3n+1}+bt^{3n+2},t^{3n+1}+at^{3n+2}+bt^{3n+3})$ are irreducible ideals cofinal with powers of the maximal ideal. By Theorem \ref{thm:finintideal}, this implies that for any ideal $I$ of $R$, the tight interior of $I$ is equal to 
\[\bigcap_{n \ge 1} J_n:(J_n:I)^*_R.\] 
Since $R$ is a 1-dimensional domain with infinite residue field, it is known \cite[Example 1.6.2]{HuTC} that  $(t^m,t^{m+1},t^{m+2})=(t^m)^-=(t^{m})^*=(t^m,t^{m+1})^*$ for $m \geq 3$. 
Using Theorem \ref{thm:finintideal},
\[
(t^m+at^{m+1}+bt^{m+2})_* =\bigcap_{n \ge 1} J_n:(J_n:(t^m+at^{m+1}+bt^{m+2}))^*.\\
\]
We note that 
\begin{align*}
   J_n:(t^m,t^{m+1},t^{m+2}) & = (J_n:t^m) \cap (J_n:t^{m+1}) \cap (J_n:t^{m+2})\\
   &=J_{n-m} \cap J_{n-m-1} \cap J_{n-m-2}=(t^{n-m+3},t^{n-m+4},t^{n-m+5}).\\
\end{align*}
For $3 \leq m<3n-2$, \[J_n:(t^m+at^{m+1}+bt^{m+2})=(t^{3n-m},t^{3n-m+1}),\] and given that the intersection is over ideals that decrease as $n$ increases  the above is equal to
\begin{align*}
\bigcap_{n \ge 1} J_n:(t^{3n-m},t^{3n-m+1})^*
&=\bigcap_{n \ge 1} J_n:(t^{3n-m},t^{3n-m+1},t^{3n-m+2})\\
&=\bigcap_{n \ge 1} (t^{m+3},t^{m+4}, t^{m+5})=(t^{m+3},t^{m+4},t^{m+5}).
\end{align*}
Thus, $(t^m+at^{m+1}+bt^{m+2})_*=(t^{m+3},t^{m+5},t^{m+6})$ which agrees with the the computations in \cite{Va-inthull}.

Similarly we can compute the $*$-interior of $(t^{m},t^{m+1},t^{m+2})$ for $m \geq 3$, using $J_n=(t^{3n},t^{3n+1})$.
\begin{align*}
(t^m,t^{m+1},t^{m+2})_*&=\bigcap_{n \ge 1} J_n:(J_n:(t^m, t^{m+1},t^{m+2}))^*\\
&=\bigcap_{n \ge 1} J_n:(t^{3n-m+3},t^{3n-m+4},t^{3n-m+5})\\
&= \bigcap_{n \ge 1} (t^{m},t^{m+1},t^{m+2})\\
&=(t^{m},t^{m+1},t^{m+2}).\\
\end{align*}

The nonzero, non-unital ideals of $R$ are of the form $(t^m,t^{m+1},t^{m+2})$, generated by two binomials whose degrees differ by at most 2, or $(t^m+at^{m+1}+bt^{m+2})$ where $m \ge 3$ and $a,b \in k$.  Hence $(t^m+at^{m+1}+bt^{m+2})$ is a $*$-expansion of $(t^{m+3},t^{m+4},t^{m+5})$.
Note that $(t^m,t^{m+1},t^{m+2}) \supseteq (t^m+at^{m+1}+bt^{m+2})$, but $(t^m,t^{m+1},t^{m+2})_*=(t^m,t^{m+1},t^{m+2})$. A maximal $*$-expansion of $(t^{m+3},t^{m+4},t^{m+5})$ is at most an ideal $I$ generated by two binomials satisfying \[(t^m+at^{m+1}+bt^{m+2}) \subseteq I \subseteq (t^m,t^{m+1},t^{m+2}).\] We have
\[\sum\limits_{a,b \in k} (t^m+at^{m+1}+bt^{k+2})=(t^{m},t^{m+1},t^{m+2}),\]
so summing over such $I$,
\[\sum\limits_{I}I=(t^m,t^{m+1},t^{m+2}),\]
which implies \[*\hull(t^{m+3},t^{m+4},t^{m+5})=(t^m,t^{m+1},t^{m+2}).\]
\end{example}

\begin{example}
Suppose $k$ is an $F$-finite field of characteristic $p>0$.  Let $R=k[[x,y]]/(xy)$.  The nonzero, non-unital ideals in $R$ are of the form $(x^n), (y^m)$, $(x^n+ay^m)$ for some nonzero $a \in k$, or $(x^n,y^m)$.  Note that for various choices of positive gradings of $x$ and $y$ in $k[x,y]$, each of these ideals is extended from a homogeneous ideal of $k[x,y]$.  Hence, by Theorem~\ref{thm:LScases}(\ref{it:graded}), the tight interior and the Artinistic tight interior of any ideal are the same, so this example could also be computed using Theorem~\ref{thm:finintideal}.

We compute the $*\hull$s of the ideals $(x^n,y^m)$.
Note that $(x^n)^*=(x^n)=(x^n)_*$, $(y^m)^*=(y^m)=(y^m)_*.$ and  $(x^n, y^m)^*=(x^n,y^m)=(x^n,y^m)_*$ by \cite[Theorem 1.3(c)]{HuTC} and \cite[Proposition 2.8]{nmeSc-tint}. 
However, 
\[(x^n+ay^m)^*=(x^n,y^m) \text{ and } (x^n+ay^m)_*=(x^{n+1},y^{m+1})\] again by \cite[Theorem 1.3(c)]{HuTC} and \cite[Proposition 2.8]{nmeSc-tint}.  
So $(x^n+ay^m)$ is a $*$-expansion of $(x^{n+1},y^{m+1})$ for all nonzero $a \in k$. Part of the lattice of ideals for $k[[x,y]]/(xy)$ includes:
$$
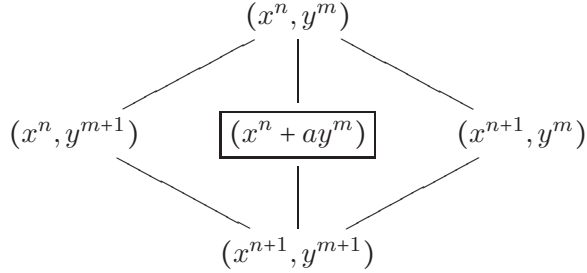
\xymatrix{ & (x^n,y^m) \ar @{-} [dl] \ar @{-} [d] \ar @{-} [dr] &\\
            (x^n,y^{m+1})\ar @{-} [dr] &\boxed{(x^n+ay^m)}\ar @{-} [d] & (x^{n+1},y^m) \ar @{-} [dl]\\
             & (x^{n+1},y^{m+1}) & \\}$$
where the boxed node contains $|k|-1$ incomparable ideals.

In fact, because  $(x^n,y^m)_*=(x^n,y^m)$, we see that the ideals $(x^n+ay^m)$ are maximal $*$-expansions of $(x^{n+1},y^{m+1})$. 
Thus, \[*\hull (x^{n+1},y^{m+1})=\sum\limits_{a \in k\backslash \{0\}} (x^n+ay^m) =(x^{n},y^{m}).\]
\end{example}

\begin{example}
Let $R=k[[x,y,z]]/(x^3+y^3+z^3)$, where $k$  is an $F$-finite field of characteristic $p>3$. 
The goal of this example is to show that the $F$-interior and $F$-hull of an ideal can vary depending on the characteristic of $k$.  By Fedder's $F$-purity criterion \cite{Fe-Fpure},
$R$ is $F$-pure if and only if $(x^3+y^3+z^3)^{p-1} \notin \m^{[p]}$ which is true if and only if $p \equiv 1 \text{ mod } 3$. 

First we will compute the $F$-interior of $(y,z)$. 
Note that 
$$(y^s,z^t)^F =\begin{cases} (y^s,z^t) \text{ if } p \equiv 1 \text{ mod } 3 \\
(x^2y^{s-1}z^{t-1},y^s,z^t)=(y^s,z^t)^* \text{ if } p \equiv 2 \text{ mod } 3 \\
\end{cases}$$
by \cite[Theorem 5.21(c)]{HoRo-purity} and \cite[Proposition 1.4]{Moira-Fclos}.

We claim that
$$(y,z)_F =\begin{cases} (y,z) \text{ if } p \equiv 1 \text{ mod } 3 \\
(xy,xz,y^2,yz,z^2) \text{ if } p \equiv 2 \text{ mod } 3. \\
\end{cases}$$

We compute the $F$-interior for the second case above using the methods from Section \ref{sec:finint}. Since $R$ is Gorenstein any system of parameters is irreducible.  Let $J_t=(y^t,z^t)$ and $I=(y,z)$.  Then 
\[J_t:(J_t:I)^F=J_t:(y^t,z^t,(yz)^{t-1})^F.\]

Note that $(y^t,z^t,y^{t-1}z^{t-1})=(y^t, z^{t-1}) \cap (y^{t-1},z^t)$.  
The test ideal of $R$ is the maximal ideal \cite[Proposition 1.4]{Moira-Fclos}.
Now by \cite[Proposition 2.4]{Va-*full}
\begin{align*}
(y^t,z^t, (yz)^{t-1})^* &=(y^t,z^t, (yz)^{t-1}):\m\\
&= (y^t,z^t,(yz)^{t-1},x^2y^{t-1}z^{t-2},x^2y^{t-2}z^{t-1}).\\
\end{align*}

When $p \equiv 2 \text{ mod } 3$, $(y^t,z^t, (yz)^{t-1})^F=(y^t,z^t, (yz)^{t-1})^*$ by \cite[Propostion 2.1]{Moira-Fclos}.
Hence, 
\begin{align*}(y,z)_F&=J_t:(y^t,z^t,(yz)^{t-1})^F\\
&=J_t:(y^t,z^t,(yz)^{t-1},x^2y^{t-1}z^{t-2},x^2y^{t-2}z^{t-1})\\
&=(xy,xz,y^2,yz,z^2)\\
\end{align*}
when $p \equiv 2 \text{ mod } 3.$

Next we compute the tight interior of the parameter ideal $(y+x^2,z)$ for $p \equiv 2 \text{ mod } 3$, with the goal of finding its Frobenius interior. For this, note first that since $R$ is an isolated singularity (as follows easily from the Jacobian criterion on the uncompleted affine ring, $k[x,y,z]/(x^3+y^3+z^3)$) 
Theorem~\ref{thm:LScases}(\ref{it:isosing}) guarantees that for any ideal $I$, the tight interior and the Artinistic tight interior of $I$ coincide. Using the same argument above with $J_t=((y+x^2)^t,z^t)$ and $I=(y+x^2,z)$ we obtain 
\begin{align*}(y&+x^2,z)_*=J_t:((y+x^2)^t,z^t,(y+x^2)^{t-1}z^{t-1})^*\\
&=J_t:((y+x^2)^t,z^t,(y+x^2)^{t-1}z^{t-1},x^2(y+x^2)^{t-1}z^{t-2},x^2(y+x^2)^{t-2}z^{t-1}))\\
&=(x(y+x^2),xz,(y+x^2)^2,(y+x^2)z,z^2)
\end{align*}
where the second equality is by \cite[Proposition 2.4]{Va-*full}.
Clearly $yz \in (y+x^2,z)_*$.  Note that $x^3+z^3=-y^3$.  
We will write $y^2$ times a unit as an element of $(y+x^2,z)_*=(x(y+x^2), xz, (y+x^2)^2, z(y+x^2), z^2)$.  First we will take a combination of $(y+x^2)^2, x(y+x^2)$ and $z^2$ and simplify algebraically.
\begin{align*}(y+x^2)^2-2x^2(y+x^2)-xz^3&=y^2+2x^2y+x^4-2x^2y-2x^4-xz^3\\
&=y^2-x^4-xz^3=y^2+xy^3\\
&=y^2(1+xy).
\end{align*}
Since $1+xy$ is a unit in $R$, $y^2$ and hence $xy$ are in $(y+x^2,z)_*$.  Thus $(y+x^2,z)_*=(xy,xz,y^2,yz,z^2).$

Note now that $(y,z)$ and $(y+x^2,z)$ are both $*$-expansions of the ideal $(xy,xz,y^2,yz,z^2)$, hence 
\[(x^2,y,z)=(y,z)+(y+x^2,z) \subseteq *\hull(xy,xz,y^2,yz,z^2).\]  

We will show that $x \notin *\hull (xy,xz,y^2,yz,z^2)$.  First, we compute the tight closures of the ideals $(y-cx,z)$ with $c^3 \neq -1$. (A similar argument can be used to compute the tight closure of the ideals $(y,z-dx)$ with $d^3 \neq -1$ or $(y-cx,z-dx)$ with $c^3+d^3 \neq -1$.) As above, we can also compute the tight closure of $(y^t,(z-cx)^t,y^{t-1}(z-cx)^{t-1})$ using \cite[Proposition 2.4]{Va-*full}

Note first that $(y-cx)(y^2+cxy+c^2x^2)+z^3=y^3-c^3x^3+z^3=(-1-c^3)x^3$  
Since $c^3 \neq -1$, $-c^3-1$ is a unit, and we can see that $x^2 \in (y-cx,z): \m$.  As a consequence, since the socle of $R/(y-cx,z)$ is generated by one element, $(y-cx,z)^*=(y-cx,z):\m=(x^2,y-cx,z)$. 
By \cite[Proposition 5.2]{Moira-Fclos}, $x^2 \in (y-cx,z)^F$ and $(y-cx,z)^F=(y-cx,z)^*$.  

As above, we compute the tight interior of $(y-cx,z)$. Let  $J_t=((y-cx)^t,z^t)$ and $I=(y-cx,z)$. Then
\begin{align*}(y&-cx,z)_*=J_t:((y-cx)^t,z^t,(y-cx)^{t-1}z^{t-1})^*\\
&=J_t:((y-cx)^t,z^t,(y-cx)^{t-1}z^{t-1},x^2(y-cx)^{t-1}z^{t-2},x^2(y-cx)^{t-2}z^{t-1})\\
&=(x(y-cx),xz,(y-cx)^2,(y-cx)z,z^2)
\end{align*}
where the second equality is by \cite[Proposition 2.4]{Va-*full}.

Although $yz \in (y-cx,z)_*$, $xy,y^2 \notin (y-cx,z)_*$. 
Thus $(y-cx,z)$ is not a $*$-extension nor a $F$-extension of $(xy,xz,y^2,yz,z^2)$.  Similarly $(y,z-dx)$ and $(y-cx,z-dx)$ are not $*$-extensions nor $F$-extensions of $(xy,xz,y^2,yz,z^2).$  

Note that \[(y,z) \subseteq F\hull(xy,xz,y^2,yz,z^2) \subseteq *\hull(xy,xz,y^2,yz,z^2) =(x^2,y,z)\] for $p \equiv 2 \text{ mod } 3$. 
McDermott \cite[Theorem 3.5]{Moira-Fclos} proved that for an $\m$-primary ideal $I$, if $I^F \neq I^*$, then there is a $\mathbb{Z}_3$-graded module $M$ and irreducible submodule $N$ of $M$ with $N^* \neq N^F$. Hence it is not known that the tight closure and Frobenius closure agree on all ideals when $p \equiv 2 \mod{3}$. 
So when $p \equiv 2 \text{ mod } 3$, if $((y+x^2)^t,z^t)^F=((y+x^2)^t,z^t)^*$ for $t>>0$
 then $F\hull(xy,xz,y^2,yz,z^2)=(x^2,y,z)$.  
 
 When the characteristic is $p \equiv 1 \text{ mod } 3$, $R$ is $F$-pure and  \[F\hull(xy,xz,y^2,yz,z^2)=(xy,xz,y^2,yz,z^2).\]  Even if $((y+cx^2)^t,z^t)^F\neq ((y+cx^2)^t,z^t)^*$ for some $c \neq 0$, the $F$-hull of $(xy,xz,y^2,yz, z^2)$ will depend on whether the characteristic is congruent to 1 or $2 \mod 3$.
\end{example}

\section*{Acknowledgments}

The authors would like to thank Eleonore Faber for suggestions of rings with known MCM modules and Jooyoun Hong and Craig Huneke for discussions of conormal modules. 
We are grateful to Karl Schwede for comments that improved the presentation of the paper. We also gratefully acknowledge the AMS-Simons Travel Grant program; the grant awarded to Rebecca R.G. subsidized the travel of Janet Vassilev to George Mason University where this collaboration commenced.

\providecommand{\bysame}{\leavevmode\hbox to3em{\hrulefill}\thinspace}
\providecommand{\MR}{\relax\ifhmode\unskip\space\fi MR }
\providecommand{\MRhref}[2]{%
  \href{http://www.ams.org/mathscinet-getitem?mr=#1}{#2}
}
\providecommand{\href}[2]{#2}

\end{document}